\documentclass[reqno,english]{amsart}%
\usepackage{amsfonts,amsmath,latexsym,verbatim,amscd,mathrsfs,color,array}
\usepackage[colorlinks=true]{hyperref}
\usepackage{amsmath,amssymb,amsthm,amsfonts,graphicx,color}
\usepackage{amssymb}
\usepackage{pdfsync}
\usepackage{epstopdf}
\usepackage{cite}
\usepackage{graphicx}
\usepackage{amsmath}
\usepackage{amsfonts}%
\usepackage[a4paper]{geometry}
\providecommand{\U}[1]{\protect\rule{.1in}{.1in}}

\numberwithin{equation}{section}

\newtheorem{theorem}{Theorem}[section]
\newtheorem{corollary}{Corollary}[section]
\newtheorem{lemma}{Lemma}[section]
\newtheorem{proposition}{Proposition}[section]
\newtheorem{remark}{Remark}[section]
\newtheorem{definition}{Definition}[section]

\numberwithin{equation}{section}

\newcommand{\bbr}{\mathbb{R}}

\newcommand{\bbn}{\mathbb{N}}
\newcommand{\ve}{\varepsilon}

\newcommand{\Uj}{\mathcal{U}_{\mu_j,\xi_j}}
\newcommand{\Un}{\mathcal{U}_{\mu_n,\xi_n}}

\newcommand{\bd}{\begin{definition}}
\newcommand{\ed}{\end{definition}}

\newcommand{\br}{\begin{remark}}
\newcommand{\er}{\end{remark}}

\newcommand{\be}{\begin{equation}}
\newcommand{\ee}{\end{equation}}

\newcommand{\bc}{\begin{corollary}}
\newcommand{\ec}{\end{corollary}}
\begin{document}

\title[Blow-up solutions]{Construction of bubbling solutions of the Brezis-Nirenberg problem in general bounded domains (I): the dimensions 4 and 5}

\author[F. Li]{Fengliu Li}
\address{\noindent  School of Mathematics, China
University of Mining and Technology, Xuzhou, 221116, P.R. China }
\email{lifengliu@cumt.edu.cn}

\author[G. Vaira]{Giusi Vaira}
\address{\noindent  Dipartimento di Matematica, Universit\'a degli Studi di Bari Aldo Moro,Italy }
\email{giusi.vaira@uniba.it}

\author[J. Wei]{Juncheng Wei}
\address{\noindent Department of Mathematics, Chinese University of Hong Kong,
Shatin, NT, Hong Kong}
\email{wei@math.cuhk.edu.hk}

\author[Y. Wu]{Yuanze Wu}
\address{\noindent  School of Mathematics, Yunnan Normal University, Kunmin, 650500, P.R. China }
\email{wuyz850306@cumt.edu.cn}

\thanks{G. Vaira is partially supported by
the MUR-PRIN-P2022YFAJH ``Linear and Nonlinear PDE's: New directions and Applications", by the MUR-PRIN-2022AKNSE ``Variational and Analytic aspects of Geometric PDEs" and by the INdAM-GNAMPA project ``Fenomeni non lineari: problemi locali e non locali e loro applicazioni",  CUP E5324001950001. J. Wei is partially supported by GRF of Hong Kong entitled "New frontiers in singularity formations in nonlinear partial differential equations".  F. Liu and Y. Wu are supported by NSFC (No. 12171470).}
\begin{abstract}
In this paper, we consider the Brezis-Nirenberg problem
\begin{eqnarray*}
\left\{
\aligned
&-\Delta u=\lambda u+|u|^{\frac{4}{N-2}}u,\quad&\mbox{in}\,\, \Omega,\\
&u=0,\quad&\mbox{on}\,\, \partial\Omega,
\endaligned
\right.
\end{eqnarray*}
where $\lambda\in\mathbb{R}$, $\Omega\subset\bbr^N$ is a bounded domain with smooth boundary $\partial\Omega$ and $N\geq3$.  We prove that every eigenvalue of the Laplacian operator $-\Delta$ with the Dirichlet boundary is a concentration value of the Brezis-Nirenberg problem in dimensions $N=4$ and $N=5$ by constructing bubbling solutions with precisely asymptotic profiles via the Ljapunov-Schmidt reduction arguments.  Our results suggest that the bubbling phenomenon of the Brezis-Nirenberg problem in dimensions $N=4$ and $N=5$ as the parameter $\lambda$ is close to the eigenvalues are governed by crucial functions related to the eigenfunctions, which has not been observed yet in the literature to our best knowledge.  Moreover, as the parameter $\lambda$ is  close to the eigenvalues, there are arbitrary number of multi-bump bubbing solutions in dimension $N=4$ while, there are only finitely many number of multi-bump bubbing solutions in dimension $N=5$, which are also new findings to our best knowledge.

\vspace{3mm} \noindent{\bf Keywords:} Brezis-Nirenberg problem; Multi-bump solution; Bubbling phenomenon; Sign-changing solution; Reduction argument.

\vspace{3mm}\noindent {\bf AMS} Subject Classification 2020: 35B33; 35B40; 35B44; 35J15.%
\end{abstract}

\date{}

\maketitle

\section{Introduction}
\subsection{Background}
The Brezis-Nirenberg problem is one of the famous problems in the community of nonlinear analysis.  It reads as
\begin{eqnarray}\label{eq0001}
\left\{
\aligned
&-\Delta u=\lambda u+|u|^{\frac{4}{N-2}}u,\quad&\mbox{in}\,\, \Omega,\\
&u=0,\quad&\mbox{on}\,\, \partial\Omega,
\endaligned
\right.
\end{eqnarray}
where $\lambda\in\mathbb{R}$, $\Omega\subset\bbr^N$ is a bounded domain with smooth boundary $\partial\Omega$ and $N\geq3$.  This famous model was introduced by Brezis and Nirenberg in their celebrated paper \cite{BN}, as an analogous example of the Yamabe problem, to understand the lack of compactness of the Sobolev embedding from $H^1_0(\Omega)$ to $L^{2^*}(\Omega)$, where $L^p(\Omega)$ is the classical Lebesgue space, $2^*=\frac{2N}{N-2}$ and
\begin{eqnarray*}
H_0^1(\Omega)=\left\{u\in L^2(\Omega)\mid |\nabla u|\in L^2(\Omega)\right\}
\end{eqnarray*}
is the classical Sobolev space.
Brezis and Nirenberg proved in \cite{BN} that \eqref{eq0001} has positive solutions for every $\lambda\in (0, \lambda_1)$ in the case $N\geq4$ and there exists $\lambda_*:=\lambda_*(\Omega)>0$ such that \eqref{eq0001} has positive solutions  if $\lambda\in (\lambda_*, \lambda_1)$ in the case $N=3$ by investigating the attainment of the following variational problem
\begin{eqnarray}\label{gs}
S_\lambda=\inf_{u\in H^1_0(\Omega)\backslash\{0\}}\frac{\|u\|^2-\lambda\|u\|_2^2}{\|u\|^{2^*}_{L^{2^*}}}
\end{eqnarray}
where $\|u\|$ and $\|u\|_{L^p}$ are the standard norms in $H_0^1(\Omega)$ and $L^p(\Omega)$, respectively, and $\lambda_1:=\lambda_1(\Omega)$ is the first eigenvalue of the Laplacian operator $-\Delta$ with Dirichlet boundary condition.  If we denote
\begin{eqnarray*}
\lambda_0=\left\{\aligned
&\lambda_*,\quad&\mbox{if }\,\,  N=3,\\
&0,\quad&\mbox{if }\,\,  N\geq4,
\endaligned
\right.
\end{eqnarray*}
then $\lambda_0$ is characterized by
\begin{eqnarray*}
\lambda_0=\inf\left\{\lambda\in\bbr\mid S_{\lambda}<S\right\}
\end{eqnarray*}
where $S$ is the best constant of the Sobolev embedding.  Moreover, as a consequence of the classical Pohozaev's identity, positive solutions do not exist if $\lambda\leq 0$ and $\Omega$ is star-shaped.  The existence of positive solutions is completely understood in the radial setting ($\Omega=B$ is the unit ball), since Brezis and Nirenberg further proved in \cite{BN} that $\lambda_*=\frac{\lambda_1(B)}{4}$ and a positive solution of \eqref{eq0001} exists if and only if $\lambda\in\left(\frac{\lambda_1(B)}{4}, \lambda_1(B)\right)$.  Finding sign-changing solutions or establishing multiplicity of solutions are much more complicated, since these topics are not very close to our main concerning in this paper, we only refer the readers to \cite{AGGS2008,CC2003,CSS1986,CW2005,DS2002,DS2003,SWW2025,SZ2010} and the references therein.

\vskip0.12in

The Brezis-Nirenberg problem~\eqref{eq0001} is also one of the classical models to understand the bubbling phenomenon of nonlinear partial differential equations (PDE for short).  To our best knowledge, in answering Brezis and Peletier's open questions proposed in \cite{B1986,BP1989}, Druet in \cite{D2002} and Han in \cite{H1991} first observed the bubbling phenomenon of the positive solution of the variational problem~\eqref{gs} as $\lambda\to\lambda_0^+$ in dimension $N=3$ and dimensions $N\geq4$, respectively, which is called the one-bubble case since there is only one bubble in this study.  Alternative proofs of this one-bubble phenomenon are given by Esposito in \cite{E2004} and Rey in \cite{R} for dimension $N=3$ and dimensions $N\geq4$, respectively.  The further related studies on such one-bubble phenomenon of the positive solutions of the Brezis-Nirenberg problem~\eqref{eq0001}, for example, critical functions and energy asymptotics, can be found in \cite{FKK2020,FKK2021,FKK2024,HV2001} and the references therein.  Since the positive solution of the variational problem~\eqref{gs} will concentrate as $\lambda\to\lambda_0^+$, we prefer to call $\lambda_0$ a {\it concentration value}, as that in \cite{AGGPV}.  Besides the one-bubble phenomenon, the positive solutions of the Brezis-Nirenberg problem~\eqref{eq0001} also has the multi-bubble phenomenon as $\lambda\to\lambda_0^+$, which, to our best knowledge, is first observed by Musso and Pistoia in \cite{MP}.  The further related studies on such multi-bubble phenomenon
of of the positive solutions, for example, the stability of the Pohozaev obstruction and multi-bubble blow-up analysis, can be found in \cite{CLP2021,DL2010, KL2022-1,KL2024} and the references therein.  A significant finding in these studies is that the bubbling phenomenon of the positive solutions of the Brezis-Nirenberg problem~\eqref{eq0001} as $\lambda\to\lambda_0^+$ is governed by the Kirchhoff-Routh function, which is given by
\begin{eqnarray}\label{eqn0002}
\mathcal{K}(\pmb{\mu},\pmb{x})=\frac12\left(\sum_{j=1}^kH(x_j,x_j)\mu_j^{N-2}-\sum_{i,j=1;i\not=j}^kG(x_i,x_j)\mu_j^{\frac{N-2}{2}}\mu_i^{\frac{N-2}{2}}\right)-\sum_{j=1}^k\frac{B_N}{2}\mu_j^2
\end{eqnarray}
where
\begin{eqnarray*}
G(x, y)=\gamma_N\left(\frac{1}{|x-y|^{N-2}}-H(x, y)\right)
\end{eqnarray*}
is the Green function of the Laplace operator $-\Delta$ at the boundary $\partial\Omega$, with $\gamma_N=\frac{1}{(N-2)\omega_N}$ and $\omega_N$ the surface area of the unit sphere in $\bbr^N$,
$H(x, y)$ is the regular part of $G(x, y)$, that is, $H(x, y)$ is the unique solution of the following equation
\begin{eqnarray*}
\left\{\begin{aligned}
&-\Delta H(x, y)=0,\quad&\mbox{in}\,\, \Omega,\\
&H(x, y)=\frac{1}{|x-y|^{N-2}}, \quad&\mbox{on}\,\, \partial\Omega,
\end{aligned}\right.
\end{eqnarray*}
$B_N$ is a constant only depending on the dimension $N$, $k$ is the number of the bubbles, $\{x_j\}\subset\Omega$ are the locations of the bubbles and $\mu_j$ is the height of the $j$-th bubble.  The bubbling phenomenon of sign-changing solutions of
the Brezis-Nirenberg problem~\eqref{eq0001} as $\lambda\to\lambda_0^+$ is much richer and more complicated, we refer the readers to \cite{BEP2006-1,BEP2006-2, IP2015, MRV2024, V2015} and the references therein.

\vskip0.12in

Besides $\lambda_0$, it is well known that there are also many other concentration values of the Brezis-Nirenberg problem~\eqref{eq0001} in low dimensions $3\leq N\leq6$ in the radial setting ($\Omega=B$ is the unit ball).  Atkinson et al. first observed such phenomenon in \cite{ABP,AP} via ODE arguments.  Based on these results, Iacopetti and Pacella gave a further detailed investigation on the sign-changing solution, having two nodal regions, in \cite{IP} in the low dimensions $4\leq N\leq6$, still by ODE arguments.  In a recent paper \cite{AGGPV}, Amadori et al. gave a completed description of concentration values and bubbling phenomenon of the Brezis-Nirenberg problem~\eqref{eq0001} in the radial setting for all $N\geq3$, mainly by ODE arguments.  In particular, it is known that in the radial setting ($\Omega=B$ is a unit ball), every eigenvalue of the Laplacian operator $-\Delta$ with Dirichlet boundary condition is a concentration value of the Brezis-Nirenberg problem~\eqref{eq0001} in dimensions $N=4$ and $N=5$.  Iacopetti and Vaira partially generalized the study of the bubbling phenomenon of the Brezis-Nirenberg problem~\eqref{eq0001} away from $\lambda_0$ in dimensions $N=4$ and $N=5$ from radial setting to general bounded domains in \cite{IV}, where it has been proved that the Brezis-Nirenberg problem~\eqref{eq0001} has a one-bump bubbling solution $u_{\lambda}$ as $\lambda\to\lambda_1$ in general (symmetric) bounded domains in dimensions $N=4$ and $N=5$, where we recall that $\lambda_1:=\lambda_1(\Omega)$ is the first eigenvalue of the Laplacian operator $-\Delta$ with Dirichlet boundary condition.  Moreover, as $\lambda\to\lambda_1$,
this one-bump bubbling solution has the same asymptotic profile as that in the radial setting established in \cite{AGGPV,IP}, that is, $u_{\lambda}^+$, the positive part of $u_{\lambda}$, concentrates and blows up at a single point and has the limit profile of a “standard bubble” in $\bbr^N$(i.e. a solution of the critical problem in $\bbr^N$ , see \eqref{pblim}) while, $u_\lambda^-$, the negative part of $u_{\lambda}$, blows down uniformly and shares the shape of the positive eigenfunction associated with $\lambda_1$.  We remark that the symmetric assumption in \cite{IV} is only used to simplify the computations and can be trivially removed.
Inspired by the above facts, it is natural to ask the following questions:
\begin{enumerate}
\item[$(Q)$]\quad Does every eigenvalue of the Laplacian operator $-\Delta$ with Dirichlet boundary condition is a concentration value of the Brezis-Nirenberg problem~\eqref{eq0001} in general bounded domains in dimensions $N=4$ and $N=5$?  If so, given a concentration value of the Brezis-Nirenberg problem~\eqref{eq0001}, which is different from $\lambda_0$,
can we find some functions, which play the same role of the Kirchhoff-Routh function $\mathcal{K}(\pmb{\mu},\pmb{x})$ given by \eqref{eqn0002} for the concentration value $\lambda_0$?
\end{enumerate}
In this paper, we shall report our answer to this natural question.

\vskip0.12in

It is worth pointing out that the situation in general bounded domains in the dimension $N=6$ is much more involved.  Indeed, Pistoia and Vaira proved in \cite{PV} that the Brezis-Nirenberg problem~\eqref{eq0001} has a one-bump bubbling solution $u_{\lambda}$ as $\lambda\to\overline{\lambda}\in (0, \lambda_1)$ in generic bounded domains in dimension $N=6$.  Again, as $\lambda\to\overline{\lambda}$, this one-bump bubbling solution has the same asymptotic profile as that in the radial setting established in \cite{AGGPV,IP}, that is, $u_{\lambda}^+$, the positive part of $u_{\lambda}$, concentrates and blows up at a single point and has the limit profile of a “standard bubble” in $\bbr^N$(i.e. a solution of the critical problem in $\bbr^N$ , see \eqref{pblim}) while, $u_\lambda^-$, the negative part of $u_{\lambda}$, uniformly converges to a positive solution of the Brezis-Nirenberg problem~\eqref{eq0001}.  Pistoia and Vaira's results suggest that in constructing bubbling solutions of the Brezis-Nirenberg problem~\eqref{eq0001} in general bounded domains in dimension $N=6$, the ansatz will be quite different from that in dimensions $N=4$ and $N=5$.  Thus, if we also report our results on constructing bubbling solutions of the Brezis-Nirenberg problem~\eqref{eq0001} in general bounded domains in dimension $N=6$ in this paper, then the length of this paper will be too long according to two different construnctions.  Taking into account this, we shall report our results on constructing bubbling solutions of the Brezis-Nirenberg problem~\eqref{eq0001} in general bounded domains in the dimension $N=6$ in a forthcoming paper.  Moreover, we remark that there is no results on constructing bubbling solutions of the Brezis-Nirenberg problem~\eqref{eq0001} in general bounded domains in the dimension $N=3$.  Besides, the result in \cite{AGGPV} about the radial setting in the dimension $N=3$ implies that the ansatz in
the dimension $N=3$ will also be very different that in dimensions $N=4, 5$ and in dimension $N=6$.  We shall report our results on constructing bubbling solutions of the Brezis-Nirenberg problem~\eqref{eq0001} in general bounded domains in the dimension $N=3$ in the future.

\vskip0.2in

\subsection{Main results}
We introduce necessary notations to state our main results.  Let $\lambda_{1}<\lambda_{2}<\cdots<\lambda_{\kappa}\to+\infty$ as $\kappa\to\infty$ be the eigenvalues of the Laplacian operator $-\Delta$ with the Dirichlet boundary condition and denote the eigenspace related to $\lambda_{\kappa}$ by
\begin{eqnarray}\label{eqnnWu0001}
\Xi_{\kappa}=\bbr e_1\oplus\bbr e_2\oplus\cdots\oplus\bbr e_{m_\kappa}
\end{eqnarray}
where $m_\kappa\in\bbn$ is the multiplicity of $\lambda_{\kappa}$ and $\{e_i\}$ is an orthogonal system.  Besides, we denote the standard Aubin-Talenti bubble for every $\mu>0$ and $\xi\in\Omega$ by $\mathcal{U}_{\mu,\xi}$ while, we denote its projection into $H^1_0(\Omega)$ by $\mathcal{W}_{\mu,\xi}$, that is, $\mathcal{U}_{\mu,\xi}$ and $\mathcal{W}_{\mu,\xi}$ are the unique solutions of the following equations
\begin{eqnarray}\label{pblim}
\left\{\begin{aligned}
&-\Delta u=u^{\frac{N+2}{N-2}},\quad&\mbox{in}\,\,  \bbr^N,\\
&u>0,\quad&\mbox{in}\,\,  \bbr^N,\\
&u(\xi)=\max_{x\in\bbr^N}u(x)=\mu^{-\frac{N-2}{2}},\\
&u\in D^{1, 2}(\bbr^N)
\end{aligned}\right.
\end{eqnarray}
and
\begin{eqnarray}\label{Pre0001}
\left\{\aligned
&-\Delta w=\mathcal{U}_{\mu,\xi}^{\frac{N+2}{N-2}},\quad&\mbox{in}\,\, \Omega,\\
&w=0,\quad&\mbox{on}\,\, \partial\Omega.
\endaligned\right.
\end{eqnarray}
Moreover, it is well known that (see \cite{A}, \cite{CGS}, \cite{T}) the Aubin-Talenti bubble $\mathcal{U}_{\mu,\xi}$ is explicitly given by
\begin{eqnarray}\label{talanti}
\mathcal{U}_{\mu,\xi}=\alpha_N\mu^{\frac{N-2}{2}}\left(\frac{1}{\mu^{2}+|x-\xi|^2}\right)^{\frac{N-2}{2}}=\mu^{-\frac{N-2}{2}}\mathcal{U}_{1,0}\left(\frac{x-\xi}{\mu}\right)
\end{eqnarray}
with $\alpha_N=[N(N-2)]^{\frac{N-2}{4}}$.  Now, our main result for $N=4$ can be stated as follows.
\begin{theorem}\label{Thm0001}
Let $N=4$, $k\geq1$ and $1\leq m\leq m_{\kappa}$ where $m_{\kappa}$ is the dimension of the eigenspace $\Xi_{\kappa}$ corresponding to $\lambda_{\kappa}$.  Then the Brezis-Nirenberg problem~\eqref{eq0001} has a solution of the form
\begin{eqnarray*}
u_\lambda=\sum_{j=1}^{k}\beta_j\mathcal{W}_{\mu_j,\xi_j}+\sum_{i=1}^{m}\tau_ie_i+\varphi_\lambda
\end{eqnarray*}
as $\lambda-\lambda_{\kappa}\to0^+$, where $\lim_{\lambda-\lambda_{\kappa}\to0^+}\tau_l=0$ for all $1\leq l\leq m$, $\lim_{\lambda-\lambda_{\kappa}\to0^+}\max\{\mu_j\}=0$ and $\lim_{\lambda-\lambda_{\kappa}\to0^+}\xi_j=\xi_{j,0}$ for all $1\leq j\leq k$, and $\lim_{\lambda-\lambda_{\kappa}\to0^+}(\max\{|\tau_l|\})^{-1}\|\varphi_\lambda\|=0$.  Moreover, $\tau_l=(t_{l,0}+o(1))\tau$ for all $1\leq l\leq m$ with
\begin{eqnarray*}
\tau=(1+o(1))\frac{c_2\left(\sum_{j=1}^k\left|\sum_{l=1}^{m}t_{l,0}e_{l}(\xi_{j,0})\right|^2\right)\exp\left(-\frac{c_1\left(\sum_{j=1}^k\left|\sum_{l=1}^{m}t_{l,0}e_{l}(\xi_{j,0})\right|^2\right)^2}{\left(\left(\sum_{j=1}^k s_{j,0}^2\right)\left\|\sum_{l=1}^{m}t_{l,0}e_l\right\|_2^2\right)\ve}\right)}{\left\|\sum_{l=1}^{m}t_{l,0}e_l\right\|_2^2\ve}
\end{eqnarray*}
and $\mu_j=(s_{j,0}+o(1))\mu$, $\beta_j=-\text{sgn}\left(\sum_{l=1}^{m}t_{l,0}e_{l}(\xi_{j,0})\right)$ with
\begin{eqnarray*}
\mu=(1+o(1))\exp\left(-\frac{c_1\left(\sum_{j=1}^k\left|\sum_{l=1}^{m}t_{l,0}e_{l}(\xi_{j,0})\right|^2\right)^2}{\left(\left(\sum_{j=1}^k s_{j,0}^2\right)\left\|\sum_{l=1}^{m}t_{l,0}e_l\right\|_2^2\right)\ve}\right)
\end{eqnarray*}
and $s_{j,0}=\left|\sum_{l=1}^{m}t_{l,0}e_{l}(\xi_{j,0})\right|$ for all $1\leq j\leq k$,
where $c_1, c_2>0$ are constants that can be precisely computed, and
\begin{eqnarray*}
(\pmb{t}_0, \pmb{\xi}_0)=(t_{1,0},t_{2,0},\cdots, t_{m,0},\xi_{1,0},\xi_{2,0},\cdots,\xi_{k,0})
\end{eqnarray*}
is a solution of the variational problem
\begin{eqnarray*}
\max_{\mathbb{S}^{m-1}\times\overline{\Omega}^k}\frac{\left(\sum_{j=1}^k\left(\sum_{l=1}^{m}\nu_{l}e_{l}(\eta_{j})\right)^2\right)}{\left\|\sum_{l=1}^{m}\nu_{l}e_l\right\|_2^2},
\end{eqnarray*}
where $\mathbb{S}^{m-1}$ is the unit sphere in $\bbr^m$.
\end{theorem}

Our main result for $N=5$ can be stated as follows.
\begin{theorem}\label{Thm0002}
Let $N=5$ and $1\leq m\leq m_{\kappa}$ where $m_{\kappa}$ is the dimension of the eigenspace $\Xi_{\kappa}$ corresponding to $\lambda_{\kappa}$.  Let
$\pmb{t}_0=(t_{1,0},t_{2,0},\cdots, t_{m,0})$
be the nontrivial solution of the variational problem
\begin{eqnarray*}
\max_{\overrightarrow{\nu}\in\bbr^m}\left(\frac{1}{2}\left\|\sum_{n=1}^{m}\nu_{n}e_n\right\|_{2}^{2}-\frac{1}{2^*}\left\|\sum_{n=1}^{m}\nu_{n}e_n\right\|_{2^*}^{2^*}\right)
\end{eqnarray*}
and $n_{\kappa}$ be the number of the nodal domains of the function $\sum_{l=1}^{m}t_{l,0}e_l(x)$.  Then for every $1\leq k\leq n_\kappa$, the Brezis-Nirenberg problem~\eqref{eq0001} has a solution of the form
\begin{eqnarray*}
u_\lambda=\sum_{j=1}^{k}\beta_j\mathcal{W}_{\mu_j,\xi_j}+\sum_{i=1}^{m}\tau_ie_i+\varphi_\lambda
\end{eqnarray*}
as $\lambda-\lambda_{\kappa}\to0^-$, where $\lim_{\lambda-\lambda_{\kappa}\to0^-}\tau_l=0$ for all $1\leq l\leq m$, $\lim_{\lambda-\lambda_{\kappa}\to0^-}\max\{\mu_j\}=0$ and $\lim_{\lambda-\lambda_{\kappa}\to0^-}\xi_j=\xi_{j,0}$ for all $1\leq j\leq k$, and $\lim_{\lambda-\lambda_{\kappa}\to0^-}(\max\{|\tau_l|\})^{-1}\|\varphi_\lambda\|=0$.  Moreover, $\tau_l=(t_{l,0}+o(1))\tau$ for all $1\leq l\leq m$ with
\begin{eqnarray*}
\tau=(1+o(1))\left(\frac{\left\|\sum_{l=1}^{m}t_{l,0}e_l\right\|_2^2}{\left\|\sum_{n=1}^{m}t_{n,0}e_n\right\|_{2^*}^{2^*}}\right)^{\frac{3}{4}}|\ve|^{\frac{3}{4}},
\end{eqnarray*}
$\mu_j=(s_{j,0}+o(1))\mu$ and $\beta_j=-\text{sgn}\left(\sum_{l=1}^{m}t_{l,0}e_{l}(\xi_{j,0})\right)$ for all $1\leq j\leq k$ with
\begin{eqnarray*}
\mu=(1+o(1))\left(\frac{c_3\left(\sum_{j=1}^k\left|\sum_{l=1}^{m}t_{l,0}e_{l}(\xi_{j,0})\right|^2\right)}{\left(\sum_{j=1}^k s_{j,0}^2\right)}\right)^2\left(\frac{\left\|\sum_{l=1}^{m}t_{l,0}e_l\right\|_2^2}{\left\|\sum_{n=1}^{m}t_{n,0}e_n\right\|_{2^*}^{2^*}}\right)^{\frac{3}{2}}|\ve|^{\frac{3}{2}},
\end{eqnarray*}
where $c_3>0$ is a constant that can be precisely computed, $\{\xi_{j,0}\}$ are either local maximum points or local minimum points of $\sum_{l=1}^{m}t_{l,0}e_l(x)$ and $\pmb{s}_0=(s_{1,0},s_{2,0},\cdots, s_{m,0})$ is a nontrivial solution of the variational problem
\begin{eqnarray*}
\max_{\overrightarrow{\nu}\in(\bbr_+)^k}\left(\sum_{j=1}^{k}\left|\sum_{l=1}^{m}t_{l,0}e_l(\xi_{j,0})\right|^2\frac{\nu_{j}^2}{2}-
\frac{1}{6}\left(\sum_{j=1}^k\left|\sum_{l=1}^{m}t_{l,0}e_{l}(\xi_{j,0})\right|\nu_{j}^{3}\right)^2\right)
\end{eqnarray*}
with $(\bbr_+)^k=\{\overrightarrow{\nu}\in\bbr^k\mid \nu_j>0\text{ for all }1\leq j\leq k\}$.
\end{theorem}

\subsection{Further remarks}
Theorems~\ref{Thm0001} and \ref{Thm0002} give positive answers to the natural question~$(Q)$ since we have asserted that all eigenvalues of the Laplacian operator $-\Delta$ with Dirichlet boundary condition are concentration values of the Brezis-Nirenberg problem~\eqref{eq0001} in general bounded domains in dimensions $N=4$ and $N=5$.  Moreover, we also find the crucial functions, which play the same role of
the Kirchhoff-Routh function $\mathcal{K}(\pmb{\mu},\pmb{x})$ given by \eqref{eqn0002} for the concentration value $\lambda_0$, that governs the bubbling phenomenon of the Brezis-Nirenberg problem~\eqref{eq0001} in general bounded domains in dimensions $N=4$ and $N=5$ near every eigenvalue of the Laplacian operator $-\Delta$ with Dirichlet boundary condition.  {\it An interesting finding} is that in general bounded domains in the dimension $N=4$, the functions governing the bubbling phenomenon of the Brezis-Nirenberg problem~\eqref{eq0001}
near every eigenvalue of the Laplacian operator $-\Delta$ with Dirichlet boundary condition are coupled while, in general bounded domains in the dimension $N=5$, the functions governing the bubbling phenomenon of the Brezis-Nirenberg problem~\eqref{eq0001}
near every eigenvalue of the Laplacian operator $-\Delta$ with Dirichlet boundary condition are decoupled.

\vskip0.12in

Theorems~\ref{Thm0001} and \ref{Thm0002} are the generalizations of \cite[Theorems~1.1 and 1.2]{IV} where such bubbling solutions were only constructed for the one-bubble case and $\kappa=1$, that is, the parameter $\lambda$ is near the first eigenvalue $\lambda_1$, in general (symmetric) bounded domains.  However, since we want to construct bubbling solutions including the multi-bubble case in general bounded domains in Theorems~\ref{Thm0001} and \ref{Thm0002}, we need to modify the arguments for \cite[Theorems~1.1 and 1.2]{IV} in a nontrivial way by introducing a {\it new strategy} to solve the reduced finite-dimensional problems.
Let us briefly sketch the proof of these two theorems in what follows.  For every $k\geq1$ and $\kappa\geq1$, we first set $\beta_j=1$ or $-1$ for all $1\leq j\leq k$ and apply the strategy in \cite{IV} to reduce the construction of the solutions of the form \eqref{eqnWu0001} for $\lambda\to\lambda_{\kappa}$.  It is worth pointing out that this reduction in general bounded domains is similar to that in the symmetric domains with slightly complicated computations and estimates since also $k\geq1$ and $\kappa\geq1$ are arbitrary in our construction.  Moreover, to avoid use the inner-out gluing type argument as that in \cite{IV} for $N=5$ in this reduction, we slightly improve the estimate on the correction $\varphi_\lambda$.  We next project the whole problem to the kernel of the involved linear operator to write down the reduced finite-dimensional problems, which read as
\begin{eqnarray}\label{expan-finite}
\left\{\aligned
0=&-d_{1}\mu_j^{\frac{N-2}{2}}\sum_{l=1}^{m}\tau_l\frac{\partial e_l(\xi_j)}{\partial x_i}+h.o.t.,\quad 1\leq j\leq k,\ 1\leq l\leq m,\\
0=&\tau_l\ve\|e_l\|_2^2+\sum_{j=1}^{k}d_{2}\beta_j\mu_j^{\frac{N-2}{2}}e_l(\xi_j)+\left\langle\left|\sum_{n=1}^{m}\tau_ne_n\right|^{\frac{4}{N-2}}\left(\sum_{n=1}^{m}\tau_ne_n\right), e_l\right\rangle_{L^2}+h.o.t.,\quad 1\leq l\leq m,\\
0=&\sum_{l=1}^{m}d_{3}\mu_j^{\frac{N-4}{2}}\tau_le_l(\xi_j)+\left\{\aligned
&\lambda_\kappa d_{4}\beta_j\mu_j|\log \mu_j|+h.o.t.,\quad N=4,\\
&\lambda_\kappa d_{4}\beta_j\mu_j+h.o.t.,\quad N=5,
\endaligned\right.
\quad 1\leq j\leq k,
\endaligned\right.
\end{eqnarray}
where $\ve=\lambda-\lambda_{\kappa}$ and $d_1,d_2,d_3,d_4>0$ are constants that can be precisely computed.  We remark that our strategy, that is, projecting the whole problem to the kernel of the involved linear operator to write dowm the reduced finite-dimensional problems, is slightly simpler than that in \cite{IV} since we do not need the $C^1$-estimate on the correction $\varphi_\lambda$.  Because $k\geq1$ and $\kappa\geq1$ are arbitrary in our construction, the major difficult in solving \eqref{expan-finite} is that the limit problem of \eqref{expan-finite} is a system (see \eqref{Redp0012} for $N=4$ and \eqref{Redp0007} for $N=5$ for more details).  To solve \eqref{expan-finite}, we shall construct stable solutions of the limit system, as that in \cite{IV} (the definition of stable solutions can be found in \cite[Definition 3.5]{IV}).  Our {\it key observation} in solving \eqref{expan-finite} is that one of the equations of the limit system of \eqref{expan-finite} for $N=5$ is decoupled while, for $N=4$, one of the equations of the limit problem of \eqref{expan-finite} has a nontrivial solution with a very special form, which could help us to decouple the limit system for $N=4$.  Thus, we can use a nonlinear Gaussian elimination method, together with the variational arguments, to construct stable solutions of the limit problem of \eqref{expan-finite} to complete our proofs of Theorems~\ref{Thm0001} and \ref{Thm0002}.

\vskip0.12in

Even though Theorems~\ref{Thm0001} and \ref{Thm0002} are the generalizations of \cite[Theorems~1.1 and 1.2]{IV} from the one-bubble case to the multi-bubble case, we still find interesting phenomenon in this generalization for the bubbling solutions of the Brezis-Nirenberg problem~\eqref{eq0001} of the form
\begin{eqnarray}\label{eqnWu0001}
u_\lambda=\sum_{j=1}^{k}\beta_j\mathcal{W}_{\mu_j,\xi_j}+\sum_{i=1}^{m}\tau_ie_i+\varphi_\lambda
\end{eqnarray}
as $\lambda-\lambda_{\kappa}\to0$.  Indeed, by Theorems~\ref{Thm0001} and \ref{Thm0002}, we observed that the number of the bubbles has no restriction for $N=4$.  Thus, there are {\it arbitrary number} of multi-bump bubbing solutions of the Brezis-Nirenberg problem~\eqref{eq0001} of the form \eqref{eqnWu0001} as $\lambda-\lambda_{\kappa}\to0^+$ for any given $\kappa\geq1$ in dimension $N=4$.  Instead, the number of the bubbles will be governed by the number of critical points of the function $\sum_{l=1}^{m}t_{l,0}e_l(x)$ where $\pmb{t}_0=(t_{1,0},t_{2,0},\cdots, t_{m,0})$ is the critical point of the function
\begin{eqnarray*}
\mathcal{F}(\pmb{\nu})=\frac{1}{2}\left\|\sum_{n=1}^{m}\nu_{n}e_n\right\|_{2}^{2}-\frac{1}{2^*}\left\|\sum_{n=1}^{m}\nu_{n}e_n\right\|_{2^*}^{2^*},
\end{eqnarray*}
$\{e_i\}$ is an orthogonal system of the eigenspace associated to $\lambda_{\kappa}$ given by \eqref{eqnnWu0001} and $1\leq m\leq m_{\kappa}$ with $m_{\kappa}$ being the multiplicity of $\lambda_{\kappa}$.
Since $\mathcal{F}(\pmb{\nu})$ is a smooth function in $\mathbb{R}^{m_{\kappa}}$ with $m_{\kappa}<+\infty$ for every $\kappa\geq1$, it has only finitely many critical points, which, together with the smoothness of $\sum_{l=1}^{m}t_{l,0}e_l(x)$, implies that there are {\it only finite number} of multi-bump bubbing solutions of the form \eqref{eqnWu0001} as $\lambda-\lambda_{\kappa}\to0^-$ for any given $\kappa\geq1$ in dimension $N=5$.

\subsection{Notations}
Throughout this paper, $a\sim b$ means that $C'b\leq a\leq Cb$ and $a\lesssim b$ means that $a\leq Cb$ where $C$ and $C'$ are positive constants, which are possibly various in different places.  The inner product in $H^1_0(\Omega)$ and $L^2(\Omega)$ will be denoted by
\begin{eqnarray*}
\langle u, v\rangle:=\int_\Omega \nabla u \nabla v\, dx\quad\text{and}\quad \langle u, v\rangle_{L^2}:=\int_\Omega uv\, dx,
\end{eqnarray*}
respectively.

\section{Setting of the problem}
Let $i^*:L^{\frac{2N}{N+2}}(\Omega)\to H^1_0(\Omega)$ be the adjoint operator of the embedding $i:H^1_0(\Omega)\to L^{\frac{2N}{N+2}}(\Omega)$, namely if $v\in L^{\frac{2N}{N+2}}(\Omega)$ then $u=i^*(v)\in H^1_0(\Omega)$ is the unique solution of the equation
\begin{eqnarray*}
\left\{\begin{aligned}
&-\Delta u=v\quad&\mbox{in}\,\, \Omega,\\
&u=0\quad&\mbox{on}\,\, \partial\Omega.
\end{aligned}\right.
\end{eqnarray*}
By continuity, it follows that
\begin{eqnarray*}
\|i^*(v)\|\lesssim \|v\|_{L^{\frac{2N}{N+2}}},\quad \forall v\in L^{\frac{2N}{N+2}}(\Omega).
\end{eqnarray*}
Hence, we can rewrite the problem \eqref{eq0001}
 in the following way
\begin{equation}\label{pbr}
\left\{\begin{aligned}&u=i^*\left[f(u)+\lambda u\right],\\
& u\in H^1_0(\Omega),
\end{aligned}\right.
\end{equation}
where $f(s)=|s|^{\frac{4}{N-2}}s$.

\vskip0,12in

For every $\mu>0$ and $\xi\in\Omega$, we define
\begin{eqnarray*}
\psi^\ell_{\mu, \xi}=\left\{\begin{aligned}
&\frac{\partial\mathcal U_{\mu, \xi}}{\partial\mu}=\frac{\alpha_N(N-2)}{2}\mu^{\frac{N-4}{2}}\frac{|x-\xi|^2-\mu^2}{(\mu^2+|x-\xi|^2)^{\frac N 2}},\quad&\mbox{if}\,\,  \ell=0,\\
&\frac{\partial\mathcal U_{\mu, \xi}}{\partial x_\ell}=-\alpha_N(N-2)\mu^{\frac{N-2}{2}}\frac{x_\ell-\xi_\ell}{(\mu^2+|x-\xi|^2)^{\frac N 2}},\quad&\mbox{if}\,\,  1\leq\ell\leq N.
\end{aligned}\right.
\end{eqnarray*}
Then it is well known (see \cite{BE1991}) that $\psi^\ell_{\mu, \xi}$ are all the solutions of the following linearized problem
\begin{eqnarray*}
\left\{\begin{aligned}
&-\Delta\psi=\frac{N+2}{N-2}\mathcal U_{\mu, \xi}^{\frac{4}{N-2}}\psi,\quad& \mbox{in }\,\, \bbr^N,\\
&\psi\in D^{1,2}(\mathbb{R}^N),
\end{aligned}\right.
\end{eqnarray*}
where $D^{1,2}(\bbr^N)=\dot{W}^{1,2}(\bbr^N)$ is the usual homogeneous Sobolev space (cf. \cite[Definition~2.1]{FG2021}).
Let us define the projection of $\psi^\ell_{\mu, \xi}$ by $P\psi^\ell_{\mu, \xi}$, that is $P\psi^\ell_{\mu, \xi}$ is the unique solution of the following equation
\begin{eqnarray}\label{projpsi-eqn}
\left\{\begin{aligned}
&-\Delta \psi=\frac{N+2}{N-2}\mathcal U_{\mu, \xi}^{\frac{4}{N-2}}\psi^\ell_{\mu, \xi}\quad&\mbox{in}\,\,  \Omega,\\
&\psi=0\quad&\mbox{on}\,\, \partial\Omega.
\end{aligned}\right.
\end{eqnarray}
It is well known (see \cite{R}) that
\begin{eqnarray}\label{projpsi}
P\psi^\ell_{\mu, \xi}=\left\{\begin{aligned}
&\psi^0_{\mu, \xi}-\alpha_N \frac{N-2}{2}\mu^{\frac{N-4}{2}}H(x, \xi)+\mathcal O\left(\mu^{\frac{N}{2}}\right), \quad&\mbox{if}\,\, \ell=0,\\
&\psi^\ell_{\mu, \xi}-\alpha_N \mu^{\frac{N-2}{2}}\frac{\partial H}{\partial \xi_\ell}(x, \xi)+\mathcal O\left(\mu^{\frac{N+2}{2}}\right),\quad&\mbox{if}\,\,  1\leq\ell\leq N.
\end{aligned}\right.
\end{eqnarray}
Moreover,
\begin{eqnarray}\label{projpsiA}
P\psi^\ell_{\mu, \xi}=\left\{\begin{aligned}
&\frac{(N-2)\alpha_N}{2\gamma_N}\mu^{\frac{N-4}{2}} G(x,\xi)+\mathcal{R}^0_{\mu, \xi},\quad&\mbox{if}\,\, \ell=0,\\
&\frac{\alpha_N}{\gamma_N}\mu^{\frac{N-2}{2}}\frac{\partial G(x,\xi)}{\partial \xi_{l}}+\mathcal{R}^\ell_{\mu, \xi}, \quad&\mbox{if}\,\,  1\leq\ell\leq N,
\end{aligned}\right.
\end{eqnarray}
where
\begin{eqnarray*}
\left|\mathcal{R}^0_{\mu, \xi}\right|\lesssim\mu^{\frac{N}{2}}+\left\{\aligned
&\mu^{\frac{N-4}{2}}|x-\xi|^{2-N},\quad&\mbox{if}\,\,  |x-\xi|\lesssim \mu,\\
&\mu^{-\frac{N}{2}},\quad&\mbox{if}\,\,  |x-\xi|\sim \mu,\\
&\mu^{\frac{N}{2}}|x-\mu|^{-N},\quad&\mbox{if}\,\,  |x-\xi|\gtrsim \mu
\endaligned\right.
\end{eqnarray*}
and
\begin{eqnarray*}
\left|\mathcal{R}^\ell_{\mu, \xi}\right|\lesssim\mu^{\frac{N+2}{2}}+\left\{\aligned
&\mu^{\frac{N-2}{2}}|x-\xi|^{1-N},\quad&\mbox{if}\,\,  |x-\xi|\lesssim \mu,\\
&\mu^{-\frac{N}{2}},\quad&\mbox{if}\,\,  |x-\xi|\sim \mu,\\
&\mu^{\frac{N+2}{2}}|x-\mu|^{-(1+N)},\quad&\mbox{if}\,\,  |x-\xi|\gtrsim \mu
\endaligned\right.
\end{eqnarray*}
for $1\leq\ell\leq N$.  We also have the following well known results (see \cite{R})
\begin{eqnarray}\label{exp}
\mathcal{W}_{\mu,\xi}:=\mathcal{U}_{\mu,\xi}-\alpha_N\mu^{\frac{N-2}{2}}H(x, \xi)+\mathcal O\left(\mu^{\frac{N+2}{2}}\right)\quad \mbox{as}\ \mu\to 0
\end{eqnarray}
and
\begin{eqnarray}\label{expA}
\left|\mathcal{W}_{\mu, \xi}-\frac{\alpha_N}{\gamma_N}\mu^{\frac{N-2}{2}}G(x,\xi)\right|\lesssim \mu^{\frac{N+2}{2}}+\left\{\aligned
  &\mu^{\frac{N-2}{2}}|x-\xi|^{2-N},\quad&\mbox{if}\,\,  |x-\xi|\lesssim\mu, \\
  &\mu^{\frac{2-N}{2}}, \quad&\mbox{if}\,\, |x-\xi|\sim\mu,\\
  &\mu^{\frac{N+2}{2}}|x-\xi|^{-N},\quad&\mbox{if}\,\,  |x-\xi|\gtrsim\mu.
\endaligned\right.
\end{eqnarray}

\vskip0,12in

Let $k\in\bbn$ and $\lambda_{\kappa}:=\lambda_{\kappa}(\Omega)$ be the $\kappa$-th eigenvalue of the Laplacian operator $-\Delta$ with the Dirichlet boundary condition.  Similar to \cite{IV}, we look for a solution of the form
\begin{eqnarray}
u_\ve&&=Z_\ve+\varphi_\ve\notag\\
&&=\underbrace{\sum_{j=1}^{k}\beta_j\mathcal{W}_{\mu_j,\xi_j}}_{Z_\ve^1}+\underbrace{\sum_{i=1}^{m}\tau_ie_i}_{:=Z_\ve^2}+\varphi_\ve
\label{solution}
\end{eqnarray}
where $\ve=|\lambda-\lambda_\kappa|$, $\{e_i\}$ is an orthogonal system of the eigenspace associated with $\lambda_{\kappa}$ given by \eqref{eqnnWu0001} and $\varphi_\ve$ will be a smaller correction as $\ve\to0^+$.  Clearly,
\begin{eqnarray}\label{lambda}
\lambda=\lambda_\kappa+\text{sgn}(\lambda-\lambda_\kappa)\ve.
\end{eqnarray}
We set the parameters $\pmb{\tau}$, $\pmb{\mu}$ and $\pmb{\xi}$ to be
\begin{eqnarray}\label{taui+muj}
\left\{\begin{aligned}
&\tau_i=t_i\tau\quad&\mbox{for all}\,\,  1\leq i\leq m,\\
&\mu_j=s_j\mu\quad&\mbox{for all}\,\,  1\leq j\leq k
\end{aligned}\right.
\end{eqnarray}
and $\pmb\xi\in O_\rho$, where
\begin{equation}\label{omegarho}
O_\rho=\left\{ \pmb{\xi}:=(\xi_1, \ldots, \xi_k)\in \Omega^k\mid \min_{1\leq i\leq k}\mbox{dist}(\xi_{i}, \partial \Omega) \geq 2\rho\text{ and } \min_{1\leq i\not=j\leq k}|\xi_{i} - \xi_{j}| \geq 2\rho\right\}
\end{equation}
with $\rho>0$ small and $\pmb{t}$, $\pmb{s}$ and $\pmb{\xi}$ also satisfy
\begin{enumerate}
\item[$(\bf {C})$]\quad $\lim_{\ve\to0^+}t_i=t_{i,0}\not=0$, $\lim_{\ve\to0^+}s_j=s_{j,0}>0$ and $\lim_{\ve\to0^+}\xi_j=\xi_{j,0}\in\Omega$.
\end{enumerate}
By the condition~$(\bf{C})$, we can write
\begin{eqnarray*}
\left\{\begin{aligned}
&\tau_i=(1+o(1))\tau_{i,0}\quad&\mbox{for all}\,\, 1\leq i\leq m,\\
&\mu_j=(1+o(1))\mu_{j,0}\quad&\mbox{for all}\,\,  1\leq j\leq k,
\end{aligned}\right.
\end{eqnarray*}
where $\tau_{i,0}=t_{i,0}\tau$ and $\mu_{j,0}=s_{j,0}\mu$.
We also let
\begin{eqnarray*}
\mathfrak A_0:= \left(\frac{c_1\left(\sum_{j=1}^k\left|\sum_{i=1}^{m}t_{i,0}e_{i}(\xi_{j,0})\right|^2\right)^2}{\left(\sum_{j=1}^k s_{j,0}^2\right)\left\|\sum_{i=1}^{m}t_{i,0}e_i\right\|_2^2}\right).
\end{eqnarray*}
We choose
\begin{eqnarray}\label{tau+mu}
\mu\sim\left\{\begin{aligned}
&e^{-\frac{\mathfrak A_0}{\ve}},\quad&\mbox{if}\,\, N=4,\\
&\ve^{\frac 32},\quad&\mbox{if}\,\, N=5
\end{aligned}\right.
\quad\text{and}\quad
\tau\sim\left\{\begin{aligned}
&\frac{e^{-\frac{\mathfrak A_0}{\ve}}}{\ve},\quad&\mbox{if}\,\, N=4,\\
&\ve^{\frac 34},\quad&\mbox{if}\,\, N=5.
\end{aligned}\right.
\end{eqnarray}
Thus, we have
\begin{eqnarray*}
\mu\sim\left\{\begin{aligned}
&\ve\tau,\quad&\mbox{if}\,\, N=4,\\
&\tau^2,\quad&\mbox{if}\,\, N=5.
\end{aligned}\right.
\end{eqnarray*}
The smaller correction $\varphi_\ve$ also need to satisfy a set of orthogonal conditions.  Let
\begin{eqnarray*}
\mathcal K_\ve:={\rm span}\left\{e_i, P\psi^\ell_{\mu_j, \xi_j}\mid 1\leq i\leq m, 1\leq j\leq k\text{ and }0\leq \ell\leq N\right\}.
\end{eqnarray*}
Its orthocomplement in $H^1_0(\Omega)$ is denoted by $\mathcal K_\ve^\bot$, that is,
\begin{eqnarray}\label{kbot}
\mathcal K_\ve^\bot=\left\{\varphi\in H^1_0(\Omega)\mid \left\langle\varphi, e_i\right\rangle=\left\langle\varphi, P\psi^\ell_{\mu_j, \xi_j}\right\rangle=0\text{ for all }i, j\text{ and }\ell\right\}.
\end{eqnarray}
We also denote the projections from $H^1_0(\Omega)$ onto $\mathcal K_\ve$ and $\mathcal K_\ve^\bot$ by $\Pi_\ve$ and $\Pi_\ve^\bot$, respectively.  With these notations in hands, solving \eqref{pbr} is equivalent to solving the following  two-coupled equations:
\begin{eqnarray}\label{pbr1}
\left\{\begin{aligned}
&\Pi_\ve^\bot\left\{Z_\ve+\varphi_\ve-i^*\left[f(Z_\ve+\varphi_\ve)+(\lambda_\kappa+\text{sgn}(\lambda-\lambda_\kappa)\ve)(Z_\ve+\varphi_\ve)\right]\right\}=0,\\
&\Pi_\ve\left\{Z_\ve+\varphi_\ve-i^*\left[f(Z_\ve+\varphi_\ve)+(\lambda_\kappa+\text{sgn}(\lambda-\lambda_\kappa)\ve)(Z_\ve+\varphi_\ve)\right]\right\}=0,
\end{aligned}\right.
\end{eqnarray}
where we rewrite $\lambda$ as in \eqref{lambda}.

\vskip0,12in

As pointed out in the introduction, we shall follow the strategy in \cite{IV} to use the Ljapunov-Schmidt reduction to solve the two-coupled equations~\eqref{pbr1}, that is, we first solve the projected problem in $\mathcal K_\ve^\bot$ (the first equation in \eqref{pbr1}) and then the finite dimensional one in $\mathcal K_\ve$ (the second equation in \eqref{pbr1}).  As usual, we shall apply the fix-point argument to solve the projected problem in $\mathcal K_\ve^\bot$, which requires us to rewrite it as follows:
\begin{eqnarray}\label{first}
\mathscr L_\ve(\varphi_\ve)=-\mathscr E_\ve-\mathscr N_\ve(\varphi_\ve)
\end{eqnarray}
where
\begin{eqnarray}\label{linear}
\mathscr L_\ve(\varphi_\ve)=\Pi_\ve^\bot\left\{\varphi_\ve-i^*\left[f'\left(Z_\ve^1\right)\varphi_\ve+(\lambda_\kappa+\text{sgn}(\lambda-\lambda_\kappa)\ve)\varphi_\ve\right]\right\}
\end{eqnarray}
is a linearized operator at $Z_\ve^1$,
\begin{eqnarray}
\label{error}\mathscr E_\ve:=\Pi_\ve^\bot\left\{Z_\ve-i^*\left[(\lambda_\kappa+\text{sgn}(\lambda-\lambda_\kappa)\ve) Z_\ve +f(Z_\ve)\right]\right\}
\end{eqnarray}
is the error term and
\begin{eqnarray*}\label{N}
\mathscr N_\ve(\varphi_\ve):=\Pi^\bot\left\{-i^*\left[f(Z_\ve+\varphi_\ve)-f(Z_\ve)-f'(Z_\ve^1)\varphi_\ve\right]\right\}
\end{eqnarray*}
is an higher order term.  After solving the projected problem in $\mathcal K_\ve^\bot$, then by \eqref{first}, we deduce that the function $u_\ve$ given by \eqref{solution} satisfies the following equation
\begin{eqnarray}\label{nonlin}
-\Delta u_\ve-f(u_\ve)-(\lambda_\kappa+\text{sgn}(\lambda-\lambda_\kappa)\ve) u_\ve=\sum_{j=1, \ldots, k \atop \ell=0, \ldots, N} \mathfrak c_j^\ell \mathcal U_{\mu_j, \xi_j}^{\frac{4}{N-2}}P\psi^\ell_{\mu_j, \xi_j}+\sum_{i=1}^m \mathfrak d_i e_i
\end{eqnarray}
for some real numbers $\mathfrak c_j^\ell=\mathfrak c_j^\ell(\pmb{\tau}, \pmb{\mu}, \pmb{\xi},\ve)$ and $\mathfrak d_i=\mathfrak d_i(\pmb{\tau}, \pmb{\mu}, \pmb{\xi},\ve)$.  Now, solving the finite dimensional problem $\mathcal K_\ve$ in \eqref{pbr1} is equivalent to finding
parameters $\pmb{\tau}(\ve)$, $\pmb{\mu}(\ve)$ and $\pmb{\xi}(\ve)$ such that $\mathfrak c_j^\ell=0$ and $\mathfrak d_i=0$ for all $j, \ell$ and $i$ as $\ve\to0^+$.

\vskip0,12in

We close this section by recalling a useful result, which will be frequently used in this paper.
\begin{lemma}\label{l01}
Let $N\leq 6$. Then for any $a, b>0$,
\begin{eqnarray*}
|f(a+b)-f(a)-f'(a)b|\lesssim \left(|a|^{\frac{6-N}{N-2}}|b|^2+|b|^{\frac{N+2}{N-2}}\right),
\end{eqnarray*}
and
\begin{eqnarray*}
|f(a+b)-f(a)|\lesssim\left(|a|^{\frac{4}{N-2}}|b|+|b|^{\frac{N+2}{N-2}}+|a|^{\frac{6-N}{N-2}}|b|^2\right).
\end{eqnarray*}
\end{lemma}

\section{The nonlinear projected problem}
To solve \eqref{first}, we need first compute the size of the error term $\mathscr E$ given by \eqref{error}.
\begin{lemma}\label{errore}
Let $\rho>0$ sufficiently small. Then there exists $\ve_0>0$ such that for any $\ve\in (0, \ve_0)$, any $\tau_i, \mu_j$ as in \eqref{taui+muj} and any $\pmb\xi\in\mathcal O_\rho$, it holds
\begin{eqnarray*}
\|\mathscr E_\ve\|=\left\{\aligned
&\mathcal{O}\left(e^{-\frac{\mathfrak A_0}{\ve}}\right),\quad&\mbox{if}\,\,  N=4, \\
&\mathcal{O}\left(\ve^{\frac{7}{4}}\right),\quad&\mbox{if}\,\,  N=5.
\endaligned\right.
\end{eqnarray*}
\end{lemma}
\begin{proof}
By \eqref{solution} and \eqref{error}, we have
\begin{eqnarray}\label{error01}
\|\mathscr E_\ve\|\lesssim \left\|\sum_{j=1}^k\beta_j  f\left( \mathcal{U}_{\mu_j,\xi_j}\right)-f(Z_\ve)\right\|_{L^{\frac{2N}{N+2}}}+\ve\tau +\sum_
{j=1}^k\left\| \mathcal{W}_{\mu_j,\xi_j}\right\|_{L^{\frac{2N}{N+2}}}.
\end{eqnarray}
Since $N=4,5$, by \eqref{talanti}, \eqref{exp}, \eqref{taui+muj}, \eqref{tau+mu} and the condition~$(\bf{C})$,
\begin{eqnarray*}
\| \mathcal{W}_{\mu_j,\xi_j}\|_{L^{\frac{2N}{N+2}}}
=\left\{\aligned
&\mathcal{O}\left(e^{-\frac{\mathfrak A_0}{\ve}}\right),\quad&\mbox{if}\,\, N=4, \\
&\mathcal{O}\left(\ve^{\frac{7}{4}}\right),\quad&\mbox{if}\,\, N=5.
\endaligned\right.
\end{eqnarray*}
while,
\begin{eqnarray*}
\begin{aligned}
\left\|\sum_{j=1}^k\beta_j  f\left( \mathcal{U}_{\mu_j,\xi_j}\right)-f(Z_\ve)\right\|_{L^{\frac{2N}{N+2}}}
&\lesssim\underbrace{\sum_{j=1}^k\left\|f( \mathcal{U}_{\mu_j,\xi_j})-f( \mathcal{W}_{\mu_j,\xi_j})\right\|_{L^{\frac{2N}{N+2}}}}_{(A)}\\
&+\underbrace{\left\|\sum_{j=1}^k\beta_jf(\mathcal{W}_{\mu_j,\xi_j})-f\left(\sum_{j=1}^k\beta_j \mathcal{W}_{\mu_j,\xi_j}\right)\right\|_{L^{\frac{2N}{N+2}}}}_{(B)}\\
&+\underbrace{\left\|f\left(\sum_{j=1}^k\beta_j \mathcal{W}_{\mu_j,\xi_j}\right)-f\left(\sum_{j=1}^k\beta_j \mathcal{W}_{\mu_j,\xi_j}+\sum_{i=1}^m\tau_i e_i\right)\right\|_{L^{\frac{2N}{N+2}}}}_{(C)}.
\end{aligned}
\end{eqnarray*}
By using the Taylor expansion (Lemma \ref{l01}), \eqref{exp}, \eqref{taui+muj} and the condition~$(\bf{C})$, we get that
\begin{eqnarray*}
\left|f( \mathcal{U}_{\mu_j,\xi_j})-f( \mathcal{W}_{\mu_j,\xi_j})\right|\lesssim \mu^{\frac{N-2}{2}}\mathcal{U}_{\mu_j,\xi_j}^{\frac{4}{N-2}}+\mu^{\frac{N+2}{2}}+ \mu^{N-2}\mathcal{U}_{\mu_j,\xi_j}^{\frac{6-N}{N-2}}.
\end{eqnarray*}
It follows from \eqref{talanti}, \eqref{tau+mu}, the condition~$(\bf{C})$ and direct computations that
\begin{eqnarray*}
(A)
=\left\{\aligned
&\mathcal{O}\left(e^{-\frac{2\mathfrak A_0}{\ve}}\right),\quad&\mbox{if}\,\, N=4, \\
&\mathcal{O}\left(\ve^{\frac{9}{2}}\right),\quad&\mbox{if}\,\, N=5.
\endaligned\right.
\end{eqnarray*}
Similarly, we also get (by using Lemma \ref{l01}, \eqref{taui+muj} and the condition~$(\bf{C})$ once more) that
\begin{eqnarray*}
\left|f\left(\sum_{j=1}^k\beta_j \mathcal{W}_{\mu_j,\xi_j}\right)-f\left(\sum_{j=1}^k\beta_j \mathcal{W}_{\mu_j,\xi_j}+\sum_{i=1}^m\tau_i e_i\right)\right|\lesssim \tau^2\sum_{j=1}^k \mathcal{U}_{\mu_j,\xi_j}^{\frac{6-N}{N-2}}+\tau^{\frac{N+2}{N-2}}+\tau\sum_{j=1}^k \mathcal{U}_{\mu_j,\xi_j}^{\frac{4}{N-2}}.
\end{eqnarray*}
Hence, by \eqref{talanti}, \eqref{taui+muj}, \eqref{tau+mu}, the condition~$(\bf{C})$ and direct computations again,
\begin{eqnarray*}
(C)=\left\{\aligned
&\mathcal{O}\left(e^{-\frac{\mathfrak A_0}{\ve}}\right),\quad&\mbox{if}\,\, N=4, \\
&\mathcal{O}\left(\ve^{\frac{7}{4}}\right),\quad&\mbox{if}\,\, N=5.
\endaligned\right.
\end{eqnarray*}
At the end we remark that to estimate the term $(B)$, we need to divide the whole domain $\Omega$ into two parts.  The first part is in a small neighborhood of the point $\xi_\ell$, say $B_{\frac\rho2}(\xi_\ell)$, for all $1\leq \ell\leq k$, where $\rho>0$ is a small constant given in \eqref{omegarho}.  In every $B_{\frac\rho2}(\xi_\ell)$, we use
Lemma \ref{l01} to expand the term $(B)$ as follows:
\begin{eqnarray*}
\begin{aligned}
&\left|f\left(\sum_{j=1}^k\beta_j \mathcal{W}_{\mu_j,\xi_j}\right)-\sum_{j=1}^k\beta_j f(\mathcal{W}_{\mu_j,\xi_j})\right|\\
&=\left|f\left(\beta_\ell \mathcal{W}_{\mu_\ell,\xi_\ell}+\sum_{j\neq \ell} \beta_j \mathcal{W}_{\mu_j,\xi_j}\right)-\beta_\ell f(\mathcal{W}_{\mu_\ell,\xi_\ell})-\sum_{j\neq \ell}\beta_jf(\mathcal{W}_{\mu_j,\xi_j})\right|\\
&\lesssim \sum_{j\neq \ell}\left(|\mathcal{W}_{\mu_\ell,\xi_\ell}|^{\frac{4}{N-2}} \left|\mathcal{W}_{\mu_j,\xi_j}\right|+|\mathcal{W}_{\mu_\ell,\xi_\ell}|^{\frac{6-N}{N-2}} \left|\mathcal{W}_{\mu_j,\xi_j}\right|^{2}+\left|\mathcal{W}_{\mu_j,\xi_j}\right|^{\frac{N+2}{N-2}}\right).
\end{aligned}
\end{eqnarray*}
The second part is the remaining region of $\Omega$, that is, $\Omega\backslash\cup_{\ell=1}^{k}B_{\frac\rho2}(\xi_\ell)$.  In this part, we simply estimate the term $(B)$ as
\begin{eqnarray*}
\left|f\left(\sum_{j=1}^k\beta_j \mathcal{W}_{\mu_j,\xi_j}\right)-\sum_{j=1}^k\beta_j f(\mathcal{W}_{\mu_j,\xi_j})\right|
\lesssim \sum_{\ell=1}^{k}\left|\mathcal{W}_{\mu_\ell,\xi_\ell}\right|^{\frac{N+2}{N-2}}.
\end{eqnarray*}
Now, by \eqref{talanti}, \eqref{exp}, \eqref{taui+muj}, \eqref{tau+mu}, the condition~$(\bf{C})$ and suitable but standard computations,
\begin{eqnarray*}
\begin{aligned}
(B)&=\left(\sum_{\ell=1}^k\int_{B_{\frac\rho2}(\xi_\ell)}\left|\cdots\cdots\right|^{\frac{2N}{N+2}}\, dx +\int_{\Omega\setminus \bigcup_{\ell=1}^k B_{\frac\rho2}(\xi_\ell)}\left|\cdots\cdots\right|^{\frac{2N}{N+2}}\, dx\right)^{\frac{N+2}{2N}}\\
&=\left\{\aligned&
\mathcal{O}\left(e^{-\frac{2\mathfrak A_0}{\ve}}\right),\quad&\mbox{if}\,\, N=4, \\
&\mathcal{O}\left(\ve^{\frac{9}{2}}\right),\quad&\mbox{if}\,\, N=5.
\endaligned\right.\end{aligned}
\end{eqnarray*}
The thesis follows from inserting the above various estimates into \eqref{error01}.
\end{proof}

After the computation of the size of the error term $\mathscr E_\ve$ given by \eqref{error}, we need to investigate the invertibility of the linear operator $\mathscr L_\ve$ and provides a uniform estimate on its inverse in using the fix-point argument to solve \eqref{first}, which will be stated in the next result.  Since it is standard nowadays to establish such result, we omit the proof here and only point out that for the readers, who are interest in the proof, one can procede as in \cite[Proposition~4.1]{IV}.
\begin{proposition}\label{lininv}
Let $\rho>0$ small enough.  Then there exist $\ve_0 >0$ such that for any $\ve\in (0, \ve_0)$, any $\tau_i, \mu_j$ as in \eqref{taui+muj} and any $\pmb\xi\in\mathcal O_\rho$, we have
$$
\|\varphi_\ve\|\lesssim\|\mathscr L_\ve(\varphi_\ve)\|\quad \mbox{for all }\varphi_\ve\in \mathcal K_\ve^\bot.
$$
In particular, the operator $\mathscr L_\ve$, which is given by \eqref{linear}, is invertible on $\mathcal K_\ve^\bot$ and its inverse is continuous which is uniformly as $\ve\to0^+$.
\end{proposition}
We are now in position, by using a standard fixed-point argument, to state that the nonlinear projected problem has a solution. Indeed the following result holds.
\begin{proposition}\label{fixedpoint}
Let $\rho>0$ small enough. Then there exists $\ve_0>0$ such that for any $\ve\in (0, \ve_0)$, any $\tau_i, \mu_j$ as in \eqref{taui+muj} and any $\pmb\xi\in\mathcal O_\rho$, the first equation in \eqref{pbr1} has a unique solution $\varphi_\ve\in \mathcal K_\ve^\bot$, which also satisfies
\begin{eqnarray*}\label{normvarphi}
\|\varphi_\ve\|=\left\{\aligned
&\mathcal{O}\left(e^{-\frac{\mathfrak A_0}{\ve}}\right),\quad&\mbox{if}\,\,  N=4, \\
&\mathcal{O}\left(\ve^{\frac{7}{4}}\right),\quad&\mbox{if}\,\,  N=5.
\endaligned\right.
\end{eqnarray*}
\end{proposition}
\begin{proof}
With Lemma~\ref{errore} and Proposition~\ref{lininv} in hands, the proof is standard so we omit it here.  We only point out that for the readers, who are interest in the proof, one can procede as in \cite[Proposition~5.2]{IV}.
\end{proof}

\section{The finite dimensional problem}
Let $\varphi_\ve$ be as in Proposition \ref{fixedpoint}, then we recall that the function $u_{\ve}$, which is given by \eqref{solution}, satisfies the equation~\eqref{nonlin}.  As pointed out above, our purpose now is to find
parameters $\pmb{\tau}(\ve)$, $\pmb{\mu}(\ve)$ and $\pmb{\xi}(\ve)$ such that the Lagrange multipliers in \eqref{nonlin}, that is, $\mathfrak c^h_n$ and $\mathfrak d_\ell$, are all zero as $\ve\to0^+$.  To achieve this goal, we shall first find out the sufficient condition, which ensures that $\mathfrak c^h_n$ and $\mathfrak d_\ell$ are all zero as $\ve\to0^+$.

\begin{lemma}\label{cs}
Let $\rho>0$ small enough.  If
\begin{eqnarray}\label{c}
\left\{\begin{aligned}
&\left\langle -\Delta u_\ve -f(u_\ve)-(\lambda_\kappa+\text{sgn}(\lambda-\lambda_\kappa)\ve)u_\ve, e_\ell\right\rangle_{L^2}=0\quad&\mbox{for all}\,\, 1\leq \ell\leq m,\\
&\left\langle -\Delta u_\ve -f(u_\ve)-(\lambda_\kappa+\text{sgn}(\lambda-\lambda_\kappa)\ve)u_\ve, P\psi^h_{\mu_n, \xi_n}\right\rangle_{L^2}=0\quad&\mbox{for all }\,\, 1\leq n\leq k\text{ and }0\leq h\leq N,
\end{aligned}\right.
\end{eqnarray}
then $\mathfrak d_\ell=0$ for all $1\leq\ell\leq m$ and $\mathfrak c^h_n=0$ for all $1\leq n\leq k$ and all $0\leq h\leq N$ as $\ve\to0^+$.
\end{lemma}
\begin{proof}
By the first equation of \eqref{c} and \eqref{nonlin},
we have
\begin{equation}\label{c1p}
\sum_{1\leq j\leq k\atop 0\leq h\leq N}\mathfrak c^h_j\left\langle \Uj^{\frac{4}{N-2}}P\psi^h_{\mu_j, \xi_j},e_\ell\right\rangle_{L^2}+\sum_{i=1}^m \mathfrak d_i\left\langle e_i,e_\ell\right\rangle_{L^2}=0.
\end{equation}
Since $\{e_i\}$ is an orthogonal system of the eigenspace of $\lambda_{\kappa}$ given by \eqref{eqnnWu0001}, we have that
\begin{eqnarray*}
\left\langle e_i,e_\ell\right\rangle_{L^2}=\left\{\begin{aligned} &\|e_\ell\|_{L^2}^2,\,\, &\mbox{if}\,\, i=\ell,\\
&0,\,\, &\mbox{if}\,\, i\neq\ell
\end{aligned}\right.
\end{eqnarray*}
while, further by \eqref{projpsi-eqn},
\begin{eqnarray*}
\begin{aligned}
\left\langle\Uj^{\frac{4}{N-2}}P\psi^h_{\mu_j, \xi_j},e_\ell\right\rangle_{L^2}&=
\left\langle \Uj^{\frac{4}{N-2}}\left(P\psi^h_{\mu_j, \xi_j}-\psi^h_{\mu_j, \xi_j}\right),e_\ell\right\rangle_{L^2}+\left\langle\Uj^{\frac{4}{N-2}}\psi^h_{\mu_j, \xi_j},e_\ell\right\rangle_{L^2}\\
&=\left\langle\Uj^{\frac{4}{N-2}}\left(P\psi^h_{\mu_j, \xi_j}-\psi^h_{\mu_j, \xi_j}\right),e_\ell\right\rangle_{L^2}+\frac{N-2}{N+2}\left\langle\nabla P\psi^h_{\mu_j, \xi_j},\nabla e_\ell\right\rangle_{L^2}\\
&=\left\langle\Uj^{\frac{4}{N-2}}\left(P\psi^h_{\mu_j, \xi_j}-\psi^h_{\mu_j, \xi_j}\right),e_\ell\right\rangle_{L^2}+\frac{N-2}{N+2}\lambda_\kappa\left\langle P\psi^h_{\mu_j, \xi_j}, e_\ell\right\rangle_{L^2}.
\end{aligned}
\end{eqnarray*}
By using \eqref{projpsi}, we have that
\begin{eqnarray*}
\begin{aligned}
\left\langle\Uj^{\frac{4}{N-2}}P\psi^h_{\mu_j, \xi_j},e_\ell\right\rangle_{L^2}
&=\left\{\begin{aligned}&
\mathcal O\left(\mu_j^{\frac{N-4}{2}} \left\|\Uj\right\|_{L^{\frac{4}{N-2}}}^{{\frac{N-2}{4}}}\right),\quad &\mbox{if}\,\, h=0,\\
&\mathcal O\left(\mu_j^{\frac{N-2}{2}} \left\|\Uj\right\|_{L^{\frac{4}{N-2}}}^{{\frac{N-2}{4}}}\right),\quad &\mbox{if}\,\, h=1, \ldots, N,\\
\end{aligned}\right.\\
&=\left\{\begin{aligned}
&\mathcal O\left(\mu_j^{2}|\log\mu_j|\right),\quad &\mbox{if}\,\, h=0\,\,\, \mbox{and}\,\,\, N=4,\\
&\mathcal O\left(\mu_j^{\frac 52}\right),\quad &\mbox{if}\,\, h=0\,\,\, \mbox{and}\,\,\, N=5,\\
&\mathcal O\left(\mu_j^3|\log\mu_j|\right),\quad &\mbox{if}\,\, h=1, \ldots 4\,\,\, \mbox{and}\,\,\, N=4,\\
&\mathcal O\left(\mu_j^{\frac 72}\right),\quad &\mbox{if}\,\, h=1, \ldots 5\,\,\, \mbox{and}\,\,\, N=5.\\
\end{aligned}\right.\\
\end{aligned}
\end{eqnarray*}
By using \eqref{projpsiA},
\begin{eqnarray}\label{prodeell}
\begin{aligned}
\left\langle P\psi^h_{\mu_j, \xi_j}, e_\ell\right\rangle_{L^2}
&=\left\{\begin{aligned}
&\frac{(N-2)\alpha_N}{2\gamma_N}\mu_j^{\frac{N-4}{2}}\int_\Omega G(x, \xi_j)e_\ell(x)\, dx+\mathcal O\left(\mu_j^{\frac N 2}\right),\quad & h=0,\\
&\mathcal O\left(\mu_j^{\frac{N-2}{2}}\right)\quad & h=1, \ldots, N,
\end{aligned}\right.\\
&=\left\{\begin{aligned}
&\frac{(N-2)\alpha_N}{2\lambda_\kappa}\mu_j^{\frac{N-4}{2}}e_\ell(\xi_j)+\mathcal O\left(\mu_j^{\frac N 2}\right),\quad &\mbox{if}\,\, h=0,\\
&\mathcal O\left(\mu_j^{\frac{N-2}{2}}\right),\quad &\mbox{if}\,\, h=1, \ldots, N,\\\end{aligned}\right.\\
&=\left\{\begin{aligned}
&\frac{(N-2)\alpha_N}{2\lambda_\kappa}e_\ell(\xi_j)+\mathcal O\left(\mu_j^2\right),\quad &\mbox{if}\,\, h=0\,\,\, \mbox{and}\,\, N=4,\\
&\mathcal O\left(\mu_j^{\frac 12}\right),\quad &\mbox{if}\,\, h=0\,\,\, \mbox{and}\,\, N=5,\\
&\mathcal O\left(\mu_j^{\frac{N-2}{2}}\right),\quad &\mbox{if}\,\, h=1, \ldots, N.\\\end{aligned}\right.\\
\end{aligned}
\end{eqnarray}
Inserting these estimates into \eqref{c1p}, we have
\begin{eqnarray}\label{(1)}
0=\left\{\begin{aligned}
&\mathfrak d_\ell \|e_\ell\|^2_{L^2}+\mathfrak a\sum_{j=1, \ldots, k}\mathfrak c_j^0 e_\ell(\xi_j)+\sum_{1\leq j\leq k\atop 1\leq h\leq N}\mathcal O(\mu)\mathfrak c_j^h,\quad&\mbox{if}\,\, N=4,\\
&\mathfrak d_\ell \|e_\ell\|^2_{L^2}+\sum_{1\leq j\leq k\atop 0\leq h\leq N}\mathcal O\left(\mu^{\frac 12}\right)\mathfrak c_j^h,\quad&\mbox{if}\,\, N=5,
\end{aligned}\right.
\end{eqnarray}
where $\mathfrak a>0$ is a constant.

By the second equation of \eqref{c} and \eqref{nonlin}, we also have
\begin{eqnarray}\label{c2p}
\sum_{1\leq j\leq k\atop 0\leq\ell \leq N}\mathfrak c^\ell_j\left\langle\Uj^{\frac{4}{N-2}}P\psi^{\ell}_{\mu_j, \xi_j},P\psi^h_{\mu_n, \xi_n}\right\rangle_{L^2}+\sum_{i=1}^m \mathfrak d_i\left\langle e_i,P\psi^h_{\mu_n, \xi_n}\right\rangle_{L^2}=0
\end{eqnarray}
for all $1\leq n\leq k$ and $0\leq h\leq N$.  By \eqref{talanti}, \eqref{projpsi} and direct computations, we can show that
\begin{eqnarray}\label{prodhl}
\left\langle \Uj^{\frac{4}{N-2}}P\psi^{\ell}_{\mu_j, \xi_j},P\psi^h_{\mu_n, \xi_n}\right\rangle_{L^2}
=\left\{\begin{aligned}
&\mu_n^{-2}\mathfrak a_0+\mathcal O(\mu_n^{N-2}),\quad &\mbox{if}\,\, j=n\,\, \mbox{and}\,\, h=\ell=0,\\
&\mu_n^{-2}\mathfrak a_1+\mathcal O(\mu_n^{N+2}),\quad &\mbox{if}\,\, j=n\,\, \mbox{and}\,\, h=\ell=1, \ldots, N,\\
&o(\mu^{-2}),\quad &\mbox{otherwise},
\end{aligned}\right.
\end{eqnarray}
where $\mathfrak a_0$ and $\mathfrak a_1$ are positive constants.
Inserting \eqref{prodeell} and \eqref{prodhl} into \eqref{c2p}, we have
\begin{equation}\label{(2)}
\mathfrak c^h_n+\sum_{1\leq j\not=n\leq k}o(1)\mathfrak c^h_j+\sum_{1\leq j\leq k\atop 0\leq\ell\not=h \leq N}o(1)\mathfrak c^\ell_j+\sum_{i=1}^m o(1)\mathfrak d_i=0
\end{equation}
for every $1\leq n\leq k$ and $0\leq h\leq N$.  The thesis now follows from \eqref{(1)} and \eqref{(2)}.
\end{proof}

According to Lemma~\ref{cs}, to prove that $\mathfrak d_\ell=0$ for all $1\leq\ell\leq m$ and $\mathfrak c^h_n=0$ for all $1\leq n\leq k$ and all $0\leq h\leq N$ as $\ve\to0^+$, it is sufficiently to expand the equations in \eqref{c} and find the main parts.  Let us begin with the expansion of the first equation of \eqref{c}.

\begin{lemma}\label{c1l}
Let $\rho>0$ small enough.  Then for any $1\leq \ell\leq m$,
\begin{eqnarray*}\label{c1mp}
&&\left\langle -\Delta u_\ve -f(u_\ve)-(\lambda_\kappa+\text{sgn}(\lambda-\lambda_\kappa)\ve)u_\ve, e_\ell\right\rangle_{L^2}\notag\\
&=&\text{sgn}(\lambda-\lambda_\kappa)\ve\tau_\ell\|e_\ell\|_2^2+\sum_{j=1}^{k}d_1\beta_j\mu_j^{\frac{N-2}{2}}e_\ell(\xi_j)\notag\\
&&+\left\langle f\left(\sum_{i=1}^{m}\tau_ie_i\right), e_\ell\right\rangle_{L^2}
+\left\{\aligned
&\mathcal{O}\left(\ve e^{-\frac{\mathfrak A_0}{\ve}}\right),\quad&\mbox{if}\,\, N=4,\\
&\mathcal{O}\left(\ve^{\frac{11}{4}}\right),\quad&\mbox{if}\,\, N=5,
\endaligned\right.
\end{eqnarray*}
as $\ve\to0^+$ uniformly with respect to $\pmb\xi\in \mathcal O_\rho$, where $d_1>0$ is a constant.
\end{lemma}
\begin{proof}
Since $\varphi_\ve\in\mathcal K_\ve^\bot$, where $\mathcal K_\ve^\bot$ is given by \eqref{kbot}, by \eqref{solution},
\begin{eqnarray}\label{el}
&&\left\langle -\Delta u_\ve -f(u_\ve)-(\lambda_\kappa+\text{sgn}(\lambda-\lambda_\kappa)\ve)u_\ve, e_\ell\right\rangle_{L^2}\notag\\
&=&-\underbrace{\left\langle\Delta Z_\ve+(\lambda_\kappa+\text{sgn}(\lambda-\lambda_\kappa)\ve) Z_\ve+f(Z_\ve), e_\ell\right\rangle_{L^2}}_{(A)}-\underbrace{\left\langle f(Z_\ve+\varphi_\ve)-f(Z_\ve), e_\ell\right\rangle_{L^2}}_{(B)}
\end{eqnarray}
for all $1\leq l\leq m$, where $Z_\ve$ is given by \eqref{solution}.
For the term $(B)$, we estimate it by applying the Taylor expansion (Lemma \ref{l01}) around $Z_\ve$ with \eqref{taui+muj} and the condition~$(\bf{C})$ to obtain that
\begin{eqnarray*}
\left|(B)\right|\lesssim \sum_{j=1}^k\left(\left\langle\mathcal U_{\mu_j, \xi_j}^{\frac{4}{N-2}},|\varphi_\ve|\right\rangle_{L^2}+\left\langle\mathcal U_{\mu_j, \xi_j}^{\frac{6-N}{N-2}},|\varphi_\ve|^2\right\rangle_{L^2}\right)+\tau^{\frac{4}{N-2}}\|\varphi_\ve\|+\tau^{\frac{6-N}{N-2}}\|\varphi_\ve\|^2+\|\varphi_\ve\|^{\frac{N+2}{N-2}},
\end{eqnarray*}
which, together with the H\"older inequality, \eqref{tau+mu}, Proposition~\ref{fixedpoint} and direct computations with \eqref{talanti}, implies that
\begin{eqnarray*}
(B)=\left\{\begin{aligned}
&\mathcal O\left(e^{-\frac{2\mathfrak A_0}{\ve}}\right),\quad&\mbox{if}\,\, N=4,\\
&\mathcal O\left(\ve^{\frac{11}{4}}\right),\quad&\mbox{if}\,\, N=5.
\end{aligned}\right.
\end{eqnarray*}
The estimate of the term $(A)$ is slightly complicated.  By \eqref{Pre0001} and \eqref{solution}, we divide it as follows:
\begin{eqnarray*}
\begin{aligned}
(A) &=\underbrace{\left\langle f(Z_\ve)-\sum_{j=1}^k\beta_j f(\mathcal U_{\mu_j, \xi_j}), e_\ell\right\rangle_{L^2}}_{(I)} +\underbrace{\sum_{i=1}^m\tau_i\text{sgn}(\lambda-\lambda_\kappa)\ve \left\langle e_i, e_\ell\right\rangle_{L^2}}_{(II)}\\&+\underbrace{\sum_{j=1}^k(\lambda_\kappa+\text{sgn}(\lambda-\lambda_\kappa)\ve)\beta_j \left\langle \mathcal W_{\mu_j, \xi_j}, e_\ell\right\rangle_{L^2}}_{(III)}.
\end{aligned}
\end{eqnarray*}
Now, since $\{e_i\}$ is an orthogonal system of the eigenspace of $\lambda_{\kappa}$ given by \eqref{eqnnWu0001}, we have that $(II)=\tau_\ell\text{sgn}(\lambda-\lambda_\kappa)\ve\|e_\ell\|^2_{L^2}$.  By \eqref{expA} and direct computations with  \eqref{talanti}, \eqref{exp}, \eqref{taui+muj}, \eqref{tau+mu} and the condition~$(\bf{C})$, we have
\begin{eqnarray*}
\left\langle \mathcal W_{\mu_j, \xi_j},e_\ell\right\rangle_{L^2}=\frac{\alpha_N}{\gamma_N} \mu_j^{\frac{N-2}{2}}\int_\Omega G(x, \xi_j) e_\ell(x)\, dx+\left\{\aligned
&\mathcal{O}\left(e^{-\frac{3\mathfrak A_0}{\ve}}\right),\quad & N=4,\\
&\mathcal{O}\left(\ve^{\frac{9}{2}}\right),\quad & N=5.
\endaligned\right.
\end{eqnarray*}
Hence, similar to \eqref{prodeell},
\begin{eqnarray*}
(III)=\sum_{j=1}^kd_1 \beta_j\mu_j^{\frac{N-2}{2}}e_\ell(\xi_j)+\left\{\aligned
&\mathcal{O}\left(\ve e^{-\frac{\mathfrak A_0}{\ve}}\right),\quad & N=4,\\
&\mathcal{O}\left(\ve^{\frac{13}{4}}\right),\quad & N=5,
\endaligned\right.
\end{eqnarray*}
where $d_1>0$ is a constant.  We further divide the term $(I)$ as follows:
\begin{eqnarray*}
\begin{aligned}
(I)&=\underbrace{\left\langle f\left(\sum_{j=1}^k \beta_j  \mathcal W_{\mu_j, \xi_j}+\sum_{i=1}^m \tau_i e_i\right)-f\left(\sum_{j=1}^k \beta_j  \mathcal W_{\mu_j, \xi_j}\right)-f\left(\sum_{i=1}^m \tau_i e_i\right), e_\ell\right\rangle_{L^2}}_{(I_1)}\\
&+\underbrace{\left\langle f\left(\sum_{j=1}^k \beta_j  \mathcal W_{\mu_j, \xi_j}\right)-\sum_{j=1}^k \beta_j  f\left(\mathcal W_{\mu_j, \xi_j}\right), e_\ell\right\rangle_{L^2}}_{(I_2)}+\underbrace{\left\langle\sum_{j=1}^k \beta_j \left(f\left( \mathcal W_{\mu_j, \xi_j}\right)-f\left(\mathcal U_{\mu_j, \xi_j}\right)\right), e_\ell\right\rangle_{L^2}}_{(I_3)}\\
&+\underbrace{\left\langle f\left( \sum_{i=1}^m \tau_i e_i\right), e_\ell\right\rangle_{L^2}}_{(I_4)}.
\end{aligned}
\end{eqnarray*}
Clearly, there is nothing to estimate for the term $(I_4)$.  For the term $(I_3)$, we can use the Taylor expansion (Lemma \ref{l01}) and direct computations with  \eqref{talanti}, \eqref{exp}, \eqref{taui+muj}, \eqref{tau+mu} and the condition~$(\bf{C})$ to obtain that
\begin{eqnarray*}
\left|(I_3)\right|&\lesssim&\sum_{j=1}^k\left(\mu_j^{\frac{N-2}{2}}\left\|\mathcal U_{\mu_j, \xi_j}\right\|_{L^{\frac{4}{N-2}}}^{\frac{N-2}{4}}+\mu_j^{\frac{N+2}{2}}+\mu_j^{N-2}\left\|\mathcal U_{\mu_j, \xi_j}\right\|_{L^{\frac{6-N}{N-2}}}^{\frac{N-2}{6-N}}\right)\\
&=&\left\{\aligned
&\mathcal{O}\left(\frac{e^{-\frac{3\mathfrak A_0}{\ve}}}{\ve}\right), &\mbox{if}\,\, N=4,\\
&\mathcal{O}\left(\ve^{\frac{21}{4}}\right), & \mbox{if}\,\, N=5.
\endaligned\right.
\end{eqnarray*}
We use the same idea of the estimate of the term $(B)$ in the proof of Lemma~\ref{errore} to estimate the term $(I_1)$.  Here, we choose the small neighborhood of every point $\xi_\ell$ to be $B_{\sqrt{\mu_\ell}}(\xi_\ell)$.  In every $B_{\sqrt{\mu_n}}(\xi_\ell)$, we rewrite
\begin{eqnarray*}
&&f\left(\sum_{j=1}^k \beta_j  \mathcal W_{\mu_j, \xi_j}+\sum_{i=1}^m \tau_i e_i\right)-f\left(\sum_{j=1}^k \beta_j  \mathcal W_{\mu_j, \xi_j}\right)-f\left(\sum_{i=1}^m \tau_i e_i\right)\\
&=&\left[f\left(\beta_n\mathcal W_{\mu_n, \xi_n}+\sum_{j\neq n}\beta_j \mathcal W_{\mu_j, \xi_j}+\sum_{i=1}^m \tau_i e_i\right)-f\left(\beta_n \mathcal W_{\mu_n, \xi_n}\right)\right]\\
&&+\left[f\left(\beta_n\mathcal W_{\mu_n, \xi_n}\right)-f\left(\beta_n \mathcal W_{\mu_n, \xi_n}+\sum_{j\neq n}\beta_j  \mathcal W_{\mu_j, \xi_j}\right)\right]-f\left(\sum_{i=1}^m \tau_i e_i\right),
\end{eqnarray*}
which, together with the Taylor expansion (Lemma \ref{l01}), implies that
\begin{eqnarray*}
&&\left|f\left(\sum_{j=1}^k \beta_j  \mathcal W_{\mu_j, \xi_j}+\sum_{i=1}^m \tau_i e_i\right)-f\left(\sum_{j=1}^k \beta_j  \mathcal W_{\mu_j, \xi_j}\right)-f\left(\sum_{i=1}^m \tau_i e_i\right)\right|\\
&\lesssim&\sum_{j\neq n}\left(\left|\mathcal W_{\mu_n, \xi_n}\right|^{\frac{4}{N-2}}\mathcal W_{\mu_j, \xi_j}+\left|\mathcal W_{\mu_j, \xi_j}\right|^{\frac{N+2}{N-2}}\right)+\tau\left|\mathcal W_{\mu_n, \xi_n}\right|^{\frac{4}{N-2}}+\tau^{\frac{N+2}{N-2}}.
\end{eqnarray*}
In the remaining region $\Omega\backslash\cup_{\ell=1}^{k}B_{\sqrt{\mu_\ell}}(\xi_\ell)$, we rewrite
\begin{eqnarray*}
&&f\left(\sum_{j=1}^k \beta_j  \mathcal W_{\mu_j, \xi_j}+\sum_{i=1}^m \tau_i e_i\right)-f\left(\sum_{j=1}^k \beta_j  \mathcal W_{\mu_j, \xi_j}\right)-f\left(\sum_{i=1}^m \tau_i e_i\right)\\
&=&\left[f\left(\sum_{j=1}^k\beta_j  \mathcal W_{\mu_j, \xi_j}+\sum_{i=1}^m \tau_i e_i\right)-f\left(\sum_{i=1}^m \tau_i e_i\right)\right]-f\left(\sum_{j=1}^k \beta_j  \mathcal W_{\mu_j, \xi_j}\right),
\end{eqnarray*}
which, together with the Taylor expansion (Lemma \ref{l01}), implies that
\begin{eqnarray*}
&&\left|f\left(\sum_{j=1}^k \beta_j  \mathcal W_{\mu_j, \xi_j}+\sum_{i=1}^m \tau_i e_i\right)-f\left(\sum_{j=1}^k \beta_j  \mathcal W_{\mu_j, \xi_j}\right)-f\left(\sum_{i=1}^m \tau_i e_i\right)\right|\\
&\lesssim&\sum_{j=1}^k \left(\tau^{\frac{4}{N-2}}\mathcal W_{\mu_j, \xi_j}+\tau^{\frac{6-N}{N-2}}\left|\mathcal W_{\mu_j, \xi_j}\right|^2+\left|\mathcal W_{\mu_j, \xi_j}\right|^{\frac{N+2}{N-2}}\right).
\end{eqnarray*}
Now, by direct computations with \eqref{talanti}, \eqref{exp}, \eqref{taui+muj}, \eqref{tau+mu} and the condition~$(\bf{C})$ in the above estimates, we get that
\begin{eqnarray*}
(I_1)&=&\int_{\bigcup_{n=1}^k B_{\sqrt{\mu_n}}(\xi_n)}\left[\ldots\ldots\right]\, dx +\int_{\Omega\setminus\bigcup_{n=1}^k B_{\sqrt{\mu_n}}(\xi_n) }\left[\ldots\ldots\right]\, dx\\
&=&\left\{\aligned
&\mathcal{O}\left(\frac{e^{-\frac{3\mathfrak A_0}{\ve}}}{\ve^2}\right), &\mbox{if}\,\, N=4,\\
&\mathcal{O}\left(\ve^{\frac{13}{4}}\right), & \mbox{if}\,\, N=5.
\endaligned\right.
\end{eqnarray*}
The terms  $(I_2)$ can be estimated in the same rate by the same way, so we omit it here.  The thesis now follows from inserting all the above estimates into \eqref{el}.
\end{proof}

Next, we expand the second equation of \eqref{c}.
\begin{lemma}\label{c3l}
Let $\rho>0$ small enough.  Then
\begin{eqnarray*}\label{c2mp}
&&\left\langle -\Delta u_\ve -f(u_\ve)-(\lambda_\kappa+\text{sgn}(\lambda-\lambda_\kappa)\ve)u_\ve, \psi^h_{\mu_n, \xi_n}\right\rangle_{L^2}\notag\\
&=&\left\{\aligned
&d_{2}\beta_n\mu_n|\log \mu_n|+\sum_{i=1}^{m}d_{3}\tau_ie_i(\xi_n)+\mathcal{O}\left(
e^{-\frac{\mathfrak A_0}{\ve}}\right),\quad& N=4,\ \ h=0\text{ and }1\leq n\leq k,\\
&d_{2}\beta_n\mu_n+\sum_{i=1}^{m}d_{3}\mu_n^{\frac{1}{2}}\tau_ie_i(\xi_n)+\mathcal{O}\left(\ve^{\frac{5}{2}}\right),\quad& N=5,\ \ h=0\text{ and }1\leq n\leq k,\\
&-d_{4}\mu_n\sum_{i=1}^{m}\tau_i\frac{\partial e_i(\xi_n)}{\partial \xi_{n, h}}+\mathcal{O}\left(e^{-\frac{2\mathfrak A_0}{\ve}}\right),\quad& N=4,\ \ 1\leq h\leq 4\text{ and }1\leq n\leq k,\\
&-d_{4}\mu_n^{\frac{3}{2}}\sum_{i=1}^{m}\tau_i\frac{\partial e_i(\xi_n)}{\partial \xi_{n, h}}+\mathcal{O}\left(\ve^{\frac{13}{4}}\right),\quad& N=5,\ \ 1\leq h\leq 5\text{ and }1\leq n\leq k
\endaligned\right.
\end{eqnarray*}
as $\ve\to0$ uniformly with respect to $\pmb\xi\in \mathcal O_\rho$, where $d_2,d_3,d_4>0$ are constants.
\end{lemma}
\begin{proof}
As that in the proof of Lemma~\ref{c1l}, we rewrite
\begin{eqnarray}\label{c3lA}
&&\left\langle -\Delta u_\ve -f(u_\ve)-(\lambda_\kappa+\text{sgn}(\lambda-\lambda_\kappa)\ve)u_\ve, \psi^h_{\mu_n, \xi_n}\right\rangle_{L^2}\notag\\
&=&-\underbrace{\left\langle \Delta Z_\ve+(\lambda_\kappa+\text{sgn}(\lambda-\lambda_\kappa)\ve) Z_\ve +f(Z_\ve), P\psi^h_{\mu_n, \xi_n}\right\rangle_{L^2}}_{(A)}\notag\\
&&-\underbrace{\left\langle f(Z_\ve+\varphi_\ve)-f(Z_\ve)-f'(Z_\ve)\varphi_\ve, P\psi^h_{\mu_n, \xi_n}\right\rangle_{L^2}}_{(B)}-(\lambda_\kappa+\text{sgn}(\lambda-\lambda_\kappa)\ve)\underbrace{\left\langle \varphi_\ve, P\psi^h_{\mu_n, \xi_n}\right\rangle_{L^2}}_{(C)}\notag\\
&&-\underbrace{\left\langle \left(f'(Z_\ve)-f'(Z_\ve^1)\right)\varphi_\ve, P\psi^h_{\mu_n, \xi_n}\right\rangle_{L^2}}_{(D)}-\underbrace{\left\langle \left(f'(Z_\ve^1)-f'(\Un)\right)\varphi_\ve, P\psi^h_{\mu_n, \xi_n}\right\rangle_{L^2}}_{(E)}\notag\\
&&-\underbrace{\left\langle f'(\Un)\left(P\psi^h_{\mu_n, \xi_n}-\psi^h_{\mu_n, \xi_n}\right),\varphi_\ve\right\rangle_{L^2}}_{(F)},
\end{eqnarray}
where $Z_\ve$ and $Z_\ve^1$ are given by \eqref{solution}.  For the term~$(A)$, by \eqref{Pre0001} and the fact that $\{e_i\}$ is an orthogonal system of the eigenspace of $\lambda_{\kappa}$ given by \eqref{eqnnWu0001}, we further rewrite
\begin{eqnarray*}
(A)&=&\underbrace{\left\langle f(Z_\ve)-\sum_{j=1}^k\beta_j f(\mathcal U_{\mu_j, \xi_j}), P\psi^h_{\mu_n, \xi_n}\right\rangle_{L^2}}_{(I)} +\underbrace{\sum_{i=1}^m\tau_i \text{sgn}(\lambda-\lambda_\kappa)\ve\left\langle e_i, P\psi^h_{\mu_n, \xi_n}\right\rangle_{L^2}}_{(II)}\notag\\
&&+\underbrace{\sum_{j=1}^k\beta_j(\lambda_\kappa+\text{sgn}(\lambda-\lambda_\kappa)\ve)\left\langle \mathcal W_{\mu_j, \xi_j},P\psi^h_{\mu_n, \xi_n}\right\rangle_{L^2}}_{(III)}.
\end{eqnarray*}
Similar to \eqref{prodeell},
\begin{eqnarray*}
(II)&=&\left\{\aligned
&\mathcal{O}\left(e^{-\frac{\mathfrak A_0}{\ve}}\right),& \mbox{if}\,\, N=4\text{ and }h=0,\\
&\mathcal{O}\left(\ve^{\frac 52}\right),& \mbox{if}\,\, N=5\text{ and }h=0,\\
&\mathcal{O}\left(e^{-\frac{2\mathfrak A_0}{\ve}}\right),& \mbox{if}\,\, N=4\text{ and }1\leq h\leq 4,\\
&\mathcal{O}\left(\ve^4\right),& \mbox{if}\,\, N=5\text{ and }1\leq h\leq 5.
\endaligned \right.
\end{eqnarray*}
For the term~$(III)$, we divide $\Omega$ into $B_{\frac\rho2}(\xi_j)$, $B_{\frac\rho2}(\xi_n)$ and $\Omega\backslash\left(B_{\frac\rho2}(\xi_j)\cup B_{\frac\rho2}(\xi_n)\right)$ for $j\not= n$.  Then by direct computations with \eqref{talanti}, \eqref{exp}, \eqref{taui+muj}, \eqref{tau+mu} and the condition~$(\bf{C})$, we have
\begin{eqnarray*}
(III)=\left\{\begin{aligned}
&d_2\beta_n\mu_n |\log \mu_n|+\mathcal{O}\left(e^{-\frac{\mathfrak A_0}{\ve}}\right),\quad &\mbox{if}\,\, N=4\text{ and }h=0,\\
&d_2\beta_n\mu_n+\mathcal{O}\left(\ve^{\frac{5}{2}}\right), \quad &\mbox{if}\,\, N=5\text{ and }h=0,\\
&\mathcal{O}\left(e^{-\frac{2\mathfrak A_0}{\ve}}\right),& \mbox{if}\,\, N=4\text{ and }1\leq h\leq 4,\\
&\mathcal{O}\left(\ve^{\frac 92}\right), & \mbox{if}\,\, N=5\text{ and }1\leq h\leq 5.
\end{aligned}\right.
\end{eqnarray*}
where $d_2>0$ is a constant.
For the term~$(I)$, we further rewrite
\begin{eqnarray*}
(I)
&=&\underbrace{\left\langle f\left(Z_\ve\right)-f\left(\sum_{j=1}^k\beta_j \mathcal W_{\mu_j, \xi_j}\right)-\sum_{i=1}^m f'\left(\sum_{j=1}^k\beta_j \mathcal W_{\mu_j, \xi_j}\right)\tau_i e_i, P\psi^h_{\mu_n, \xi_n}\right\rangle_{L^2}}_{(I_1)}\\
&&+\underbrace{\left\langle f\left(\sum_{j=1}^k\beta_j \mathcal W_{\mu_j, \xi_j}\right)-\sum_{j=1}^k\beta_j f\left(\mathcal W_{\mu_j, \xi_j}\right), P\psi^h_{\mu_n, \xi_n}\right\rangle_{L^2}}_{(I_2)}\\
&&+\underbrace{\sum_{j=1}^k\beta_j\left\langle f\left(\mathcal W_{\mu_j, \xi_j}\right)- f\left(\mathcal U_{\mu_j, \xi_j}\right), P\psi^h_{\mu_n, \xi_n}\right\rangle_{L^2}}_{(I_3)}+\underbrace{\sum_{j=1}^k\sum_{i=1}^m \tau_i \left\langle f'(\mathcal U_{\mu_j, \xi_j})e_i, P\psi^h_{\mu_n, \xi_n}\right\rangle_{L^2}}_{(I_4^1)}\\
&&+\underbrace{\sum_{i=1}^m\tau_i\left\langle\left(f'\left(\sum_{j=1}^k\beta_j \mathcal W_{\mu_j, \xi_j}\right)-\sum_{j=1}^k f'(\beta_j\mathcal W_{\mu_j, \xi_j})\right)e_i, P\psi^h_{\mu_n, \xi_n}\right\rangle_{L^2}}_{(I_4^2)}\\
&&+\underbrace{\sum_{j=1}^k\sum_{i=1}^m \tau_i \left\langle \left(f'(\beta_j\mathcal W_{\mu_j, \xi_j})-f'(\beta_j\mathcal U_{\mu_j, \xi_j})\right)e_i, P\psi^h_{\mu_n, \xi_n}\right\rangle_{L^2}}_{(I_4^3)}.
\end{eqnarray*}
For what concerning $(I_1)$, by using the Taylor expansion (Lemma \ref{l01}) around $Z^1_{\ve}$, we get that
\begin{eqnarray*}
\left|(I_1)\right|&\lesssim&\tau^2 \sum_{j=1}^k \left\langle \mathcal U_{\mu_j, \xi_j}^{\frac{6-N}{N-2}},|\psi^0_{\mu_n, \xi_n}|\right\rangle_{L^2}+\tau^{\frac{N+2}{N-2}}\|\psi^0_{\mu_n, \xi_n}\|_{L^1}.
\end{eqnarray*}
We then divide $\Omega$ into $B_{\frac\rho2}(\xi_j)$, $B_{\frac\rho2}(\xi_n)$ and $\Omega\backslash\left(B_{\frac\rho2}(\xi_j)\cup B_{\frac\rho2}(\xi_n)\right)$ for $j\not= n$, and directly estimate all the above terms with \eqref{talanti}, \eqref{exp}, \eqref{taui+muj}, \eqref{tau+mu} to obtain that
\begin{eqnarray*}
\left|(I_1)\right|
&=&\left\{\begin{aligned}
&\mathcal{O}\left(\frac{e^{-\frac{3\mathfrak A_0}{\ve}}}{\ve^3}\right),\quad&\mbox{if}\,\, N=4\text{ and }h=0,\\
&\mathcal{O}\left(\ve^{\frac 52}\right),\quad&\mbox{if}\,\, N=5\text{ and }h=0,\\
&\mathcal{O}\left(\frac{e^{-\frac{4\mathfrak A_0}{\ve}}}{\ve^3}\right),\quad&\mbox{if}\,\, N=4\text{ and }1\leq h\leq 4,\\
&\mathcal{O}\left(\ve^{4}\right),\quad&\mbox{if}\,\, N=5\text{ and }1\leq h\leq 5.
\end{aligned}\right.
\end{eqnarray*}
For what concerning $(I_2)$, we apply the same idea of the estimate of the term~$(B)$ in the proof of Lemma~\ref{errore} by dividing $\Omega$ into the small neighborhood of the point $\xi_\ell$, say $B_{\frac\rho2}(\xi_\ell)$, for all $1\leq \ell\leq k$ and the remaining region of $\Omega$, that is, $\Omega\backslash\cup_{\ell=1}^{k}B_{\frac\rho2}(\xi_\ell)$.  Then we estimate $(I_2)$ in these regions with \eqref{talanti}, \eqref{exp}, \eqref{taui+muj}, \eqref{tau+mu} and the condition~$(\bf{C})$ in a very similar way of $(I_1)$ in the proof of Lemma~\ref{c1l} to obtain that
\begin{eqnarray*}
(I_2)&=&\left\{\begin{aligned}
&\mathcal{O}\left( e^{-\frac{\mathfrak A_0}{\ve}}\right),\quad &\mbox{if}\,\, N=4\text{ and }h=0,\\
&\mathcal{O}\left(\ve^3\right), \quad &\mbox{if}\,\, N=5\text{ and }h=0,\\
&\mathcal{O}\left( e^{-\frac{2\mathfrak A_0}{\ve}}\right),\quad &\mbox{if}\,\, N=4\text{ and }\text{ and }1\leq h\leq 4,\\
&\mathcal{O}\left(\ve^{\frac{9}{2}}\right), \quad &\mbox{if}\,\, N=5\text{ and }\text{ and }1\leq h\leq 5.
\end{aligned}\right.
\end{eqnarray*}
For what concerning $(I_3)$, the estimate is similar to that of $(I_1)$, that is, we first using the Taylor expansion (Lemma \ref{l01}) around $\mathcal U_{\mu_j, \xi_j}$, next divide $\Omega$ into $B_{\frac\rho2}(\xi_j)$, $B_{\frac\rho2}(\xi_n)$ and $\Omega\backslash\left(B_{\frac\rho2}(\xi_j)\cup B_{\frac\rho2}(\xi_n)\right)$ for $j\not= n$ and then directly estimate all the above terms with \eqref{talanti}, \eqref{exp}, \eqref{taui+muj}, \eqref{tau+mu} and the condition~$(\bf{C})$ to obtain that
\begin{eqnarray*}
(I_3)
&=&\left\{\begin{aligned}
&\mathcal{O}\left( e^{-\frac{\mathfrak A_0}{\ve}}\right),\quad &\mbox{if}\,\, N=4\text{ and }h=0,\\
&\mathcal{O}\left(\ve^3\right), \quad &\mbox{if}\,\, N=5\text{ and }h=0,\\
&\mathcal{O}\left( e^{-\frac{2\mathfrak A_0}{\ve}}\right),\quad &\mbox{if}\,\, N=4\text{ and }\text{ and }1\leq h\leq 4,\\
&\mathcal{O}\left(\ve^{\frac{9}{2}}\right), \quad &\mbox{if}\,\, N=5\text{ and }\text{ and }1\leq h\leq 5.
\end{aligned}\right.
\end{eqnarray*}
We remark that more precise estimates of $(I_1)$, $(I_2)$ and $(I_3)$ can be found in \cite{MP}, which we do not need here according to our choice of $\tau$ and $\mu$ given by \eqref{tau+mu}.
For what concerning $(I_4^1)$, by direct computations with \eqref{talanti}, \eqref{exp},  \eqref{projpsi}, \eqref{taui+muj}, \eqref{tau+mu} and the condition~$(\bf{C})$, we get that
\begin{eqnarray*}
(I_4^1)&=&\sum_{i=1}^m \tau_i\left\langle f'(\mathcal U_{\mu_n, \xi_n})e_i, \left(\psi^h_{\mu_n, \xi_n}+\mathcal O\left(\mu^{\frac{N-4}{2}}\right)\right)\right\rangle_{L^2}\\
&&+\sum_{j\neq n}\sum_{i=1}^m \tau_i\left\langle f'(\mathcal U_{\mu_j, \xi_j})e_i,P\psi^h_{\mu_n, \xi_n}\right\rangle_{L^2}\\
&=&\left\{\begin{aligned}
&\sum_{i=1}^md_3\tau_ie_i(\xi_n)+\mathcal O\left(\frac{e^{-\frac{2\mathfrak A_0}{\ve}}}{\ve^2}\right),\quad &\mbox{if}\,\, N=4\text{ and }h=0,\\
&\sum_{i=1}^md_3\mu_n^{\frac12}\tau_ie_i(\xi_n)+O\left(\ve^3\right),\quad &\mbox{if}\,\, N=5\text{ and }h=0,\\
&-\left(\sum_{i=1}^{m}d_{4}\mu_n\tau_i\frac{\partial e_i(\xi_n)}{\partial \xi_{n, h}}\right)+\mathcal O\left(\frac{e^{-\frac{2\mathfrak A_0}{\ve}}}{\ve^2}\right),\quad &\mbox{if}\,\, N=4\text{ and }1\leq h\leq 4,\\
&-\left(\sum_{i=1}^{m}d_{4}\mu_n^{\frac{3}{2}}\tau_i\frac{\partial e_i(\xi_n)}{\partial \xi_{n, h}}\right)+O\left(\ve^{\frac{9}{2}}\right),\quad &\mbox{if}\,\, N=5\text{ and }1\leq h\leq 5,
\end{aligned}\right.
\end{eqnarray*}
where $d_3, d_4>0$ are constants.  For what concerning $(I_4^2)$, the estimate is very similar to that of $(I_2)$, so we omit the details here.  The resulting estimate is given by
\begin{eqnarray*}
(I_4^2)
&=&\left\{\begin{aligned}
&\mathcal{O}\left(\frac{e^{-\frac{3\mathfrak A_0}{\ve}}}{\ve^2}\right),\quad&\mbox{if}\,\, N=4\text{ and }h=0,\\
&\mathcal{O}\left(\ve^{\frac 92}\right),\quad&\mbox{if}\,\, N=5\text{ and }h=0,\\
&\mathcal{O}\left(\frac{e^{-\frac{4\mathfrak A_0}{\ve}}}{\ve^2}\right),\quad&\mbox{if}\,\, N=4\text{ and }1\leq h\leq 4,\\
&\mathcal{O}\left(\ve^{6}\right),\quad&\mbox{if}\,\, N=5\text{ and }1\leq h\leq 5.
\end{aligned}\right.
\end{eqnarray*}
For what concerning $(I_4^3)$, the estimate is very similar to that of $(I_3)$, so we also omit the details here.  The resulting estimate is given by
\begin{eqnarray*}
(I_4^3)
&=&\left\{\begin{aligned}
&\mathcal{O}\left(\frac{e^{-\frac{3\mathfrak A_0}{\ve}}}{\ve^2}\right),\quad&\mbox{if}\,\, N=4\text{ and }h=0,\\
&\mathcal{O}\left(\ve^{\frac 92}\right),\quad&\mbox{if}\,\, N=5\text{ and }h=0,\\
&\mathcal{O}\left(\frac{e^{-\frac{4\mathfrak A_0}{\ve}}}{\ve^2}\right),\quad&\mbox{if}\,\, N=4\text{ and }1\leq h\leq 4,\\
&\mathcal{O}\left(\ve^{6}\right),\quad&\mbox{if}\,\, N=5\text{ and }1\leq h\leq 5.
\end{aligned}\right.
\end{eqnarray*}
The estimates of $(B)$-$(F)$ are also similar.  We only point out that in estimating $(C)$ and $(F)$, we just first use the H\"older inequality and Proposition~\ref{fixedpoint} and then directly compute them with \eqref{talanti}, \eqref{exp},  \eqref{projpsi}, \eqref{taui+muj}, \eqref{tau+mu} and the condition~$(\bf{C})$ while, in estimating $(B)$, $(D)$ and $(E)$, we need further employing the Taylor expansion (Lemma \ref{l01}) around $Z_\ve$, $Z_\ve^1$ and $\Un$, respectively.  The resulting estimates are
\begin{eqnarray*}
(B)+(C)+(D)+(E)+(F)=\left\{\begin{aligned}
&\mathcal{O}\left( e^{-\frac{\mathfrak A_0}{\ve}}\right),\quad&\mbox{if}\,\, N=4,\\
&\mathcal{O}\left(\ve^{\frac 52}\right),\quad&\mbox{if}\,\, N=5.
\end{aligned}\right.
\end{eqnarray*}
The conclusions then follow from inserting all the above estimates into \eqref{c3lA}.
\end{proof}

\section{Solving The reduced problem: proof of the Theorems}
To complete our constructions, by Proposition~\ref{fixedpoint} and Lemmas~\ref{cs}, \ref{c1l} and \ref{c3l}, we need to solve the following system of the parameters $\pmb{t}$, $\pmb{s}$ and $\pmb{\xi}$ as $\ve\to0^+$:
\begin{eqnarray}\label{redprob}
\left\{\aligned
0=&\text{sgn}(\lambda-\lambda_\kappa)\ve\tau_l\|e_l\|_2^2+\sum_{j=1}^{k}d_{1}\beta_j\mu_j^{\frac{N-2}{2}}e_l(\xi_j)+\left\langle f\left(\sum_{n=1}^{m}\tau_ne_n\right), e_l\right\rangle_{L^2}\\
&+\left\{\aligned
&\mathcal{O}\left(e^{-\frac{2\mathfrak A_0}{\ve}}\right),\quad & N=4,\\
&\mathcal{O}\left(\ve^{\frac{11}{4}}\right),\quad & N=5,
\endaligned\right.,\quad\mbox{for all}\,\, 1\leq l\leq m,\\
0=&\sum_{l=1}^{m}d_{3}\mu_j^{\frac{N-4}{2}}\tau_le_l(\xi_j)+\left\{\aligned
&d_{2}\beta_j\mu_j|\log \mu_j|+\mathcal{O}\left(e^{-\frac{\mathfrak A_0}{\ve}}\right),\quad &N=4,\\
&d_{2}\beta_j\mu_j+\mathcal{O}\left(\ve^{\frac 52}\right),\quad&N=5
\endaligned\right.,
\quad\mbox{for all}\,\, 1\leq j\leq k,\\
0=&-\left(\sum_{l=1}^{m}d_{4}\mu_j^{\frac{N-2}{2}}\tau_l\frac{\partial e_l(\xi_j)}{\partial x_i}\right)+\left\{\aligned
&\mathcal{O}\left(e^{-\frac{2\mathfrak A_0}{\ve}}\right),\quad &N=4,\\
&\mathcal{O}\left(\ve^{\frac{13}{4}}\right),\quad& N=5.
\endaligned\right.,\quad\mbox{for all}\,\, 1\leq j\leq k,\ 1\leq l\leq m.
\endaligned\right.
\end{eqnarray}
According to \eqref{taui+muj}, \eqref{tau+mu} and the condition~$(\bf{C})$, if we take $\beta_j=-\text{sgn}\left(\sum_{l=1}^{m}t_{l,0}e_{l}(\xi_{j,0})\right)$ for all $1\leq j\leq k$,
\begin{eqnarray}\label{Redp0011}
\left\{\aligned
\mu&=(1+o(1))\exp\left(-\frac{d_1d_3L_0^2}{\left(d_2\left(\sum_{j=1}^k s_{j,0}^2\right)\left\|\sum_{l=1}^{m}t_{l,0}e_l\right\|_2^2\right)\ve}\right),\\
\tau&=(1+o(1))\frac{d_1L_0\exp\left(-\frac{d_1d_3L_0^2}{\left(d_2\left(\sum_{j=1}^k s_{j,0}^2\right)\left\|\sum_{l=1}^{m}t_{l,0}e_l\right\|_2^2\right)\ve}\right)}{\left\|\sum_{l=1}^{m}t_{l,0}e_l\right\|_2^2\ve}
\endaligned\right.
\end{eqnarray}
for $N=4$ and
\begin{eqnarray}\label{Redp0008}
\left\{\aligned
\tau&=(1+o(1))\left(\frac{\left\|\sum_{l=1}^{m}t_{l,0}e_l\right\|_2^2}{\left\|\sum_{n=1}^{m}t_{n,0}e_n\right\|_{2^*}^{2^*}}\right)^{\frac{3}{4}}\ve^{\frac{3}{4}},\\
\mu&=(1+o(1))\left(\frac{d_3L_0}{ d_2\left(\sum_{j=1}^k s_{j,0}^2\right)}\right)^2\left(\frac{\left\|\sum_{l=1}^{m}t_{l,0}e_l\right\|_2^2}{\left\|\sum_{n=1}^{m}t_{n,0}e_n\right\|_{2^*}^{2^*}}\right)^{\frac{3}{2}}\ve^{\frac{3}{2}}
\endaligned\right.
\end{eqnarray}
for $N=5$ in \eqref{redprob}, then \eqref{redprob} reduces to the following systems:
\begin{eqnarray}\label{redprobA}
\left\{\aligned
o(1)=&\sum_{l=1}^{m}t_{l}\frac{\partial e_l(\xi_{j})}{\partial x_i},\quad 1\leq j\leq k\text{ and }1\leq i\leq 4,\\
o(1)=&\left(\frac{L\|e_l\|_2^2}{\left\|\sum_{l=1}^{m}t_{l}e_l\right\|_2^2}\right)t_{l}+\sum_{j=1}^{k}\text{sgn}\left(\sum_{l=1}^{m}t_{l}e_{l}(\xi_{j})\right)s_{j}e_l(\xi_{j}),
\quad 1\leq l\leq m,\\
o(1)=&\sum_{l=1}^{m}\left(\left(\sum_{i=1}^k s_{i}^2\right)e_l(\xi_{j})\right)t_{l}+\text{sgn}\left(\sum_{l=1}^{m}t_{l}e_{l}(\xi_{j})\right)Ls_{j},
\quad 1\leq j\leq k
\endaligned\right.
\end{eqnarray}
for $N=4$ and
\begin{eqnarray}\label{redprobB}
\left\{\aligned
o(1)=&\sum_{l=1}^{m}t_{l}\frac{\partial e_l(\xi_{j})}{\partial x_i},\quad 1\leq j\leq k\text{ and }1\leq i\leq 5,\\
o(1)=&t_{l}\|e_l\|_2^2-\left(\frac{\left\|\sum_{l=1}^{m}t_{l}e_l\right\|_2^2}{\left\|\sum_{n=1}^{m}t_{n}e_n\right\|_{2^*}^{2^*}}\right)\left\langle f\left(\sum_{n=1}^{m}t_{n}e_n\right), e_l\right\rangle_{L^2}, 1\leq l\leq m,\\
o(1)=&\sum_{l=1}^{m}t_{l}e_l(\xi_j)+
\frac{\text{sgn}\left(\sum_{l=1}^{m}t_{l}e_{l}(\xi_{j})\right)L}{\sum_{j=1}^k s_{j}^2}s_{j}^{\frac12},
\quad 1\leq j\leq k
\endaligned\right.
\end{eqnarray}
for $N=5$, where
\begin{eqnarray*}
L=-\left(\sum_{j=1}^k\left|\sum_{l=1}^{m}t_{l}e_{l}(\xi_{j})\right|s_{j}^{\frac{N-2}2}\right).
\end{eqnarray*}
Moreover,
\begin{eqnarray}\label{sgnlambda}
\left\{\aligned
&\text{sgn}(\lambda-\lambda_\kappa)>0,\quad&\mbox{if}\,\, N=4\\
&\text{sgn}(\lambda-\lambda_\kappa)<0,\quad&\mbox{if}\,\, N=5.
\endaligned\right.
\end{eqnarray}
We remark that under the condition~$(\bf{C})$, $\beta_j=-\text{sgn}\left(\sum_{l=1}^{m}t_{l}e_{l}(\xi_{j})\right)$ for all $1\leq j\leq k$ as $\ve\to0^+$.
Clearly, by the condition~$(\bf{C})$ once more, the limit systems of \eqref{redprobA} and \eqref{redprobB} are given by
\begin{eqnarray}\label{Redp0012}
\left\{\aligned
0=&\sum_{l=1}^{m}t_{l,0}\frac{\partial e_l(\xi_{j,0})}{\partial x_i},\quad 1\leq j\leq k\text{ and }1\leq i\leq 4,\\
0=&\left(\frac{L_0\|e_l\|_2^2}{\left\|\sum_{l=1}^{m}t_{l,0}e_l\right\|_2^2}\right)t_{l,0}+\sum_{j=1}^{k}\text{sgn}\left(\sum_{l=1}^{m}t_{l,0}e_{l}(\xi_{j,0})\right)s_{j,0}e_l(\xi_{j,0}),
\quad 1\leq l\leq m,\\
0=&\sum_{l=1}^{m}\left(\left(\sum_{i=1}^k s_{i,0}^2\right)e_l(\xi_{j,0})\right)t_{l,0}+\text{sgn}\left(\sum_{l=1}^{m}t_{l,0}e_{l}(\xi_{j,0})\right)L_0s_{j,0},
\quad 1\leq j\leq k
\endaligned\right.
\end{eqnarray}
for $N=4$ and
\begin{eqnarray}\label{Redp0007}
\left\{\aligned
0=&\sum_{l=1}^{m}t_{l,0}\frac{\partial e_l(\xi_{j,0})}{\partial x_i},\quad 1\leq j\leq k\text{ and }1\leq i\leq 5,\\
0=&t_{l,0}\|e_l\|_2^2-\left(\frac{\left\|\sum_{l=1}^{m}t_{l,0}e_l\right\|_2^2}{\left\|\sum_{n=1}^{m}t_{n,0}e_n\right\|_{2^*}^{2^*}}\right)\left\langle f\left(\sum_{n=1}^{m}t_{n,0}e_n\right), e_l\right\rangle_{L^2}, 1\leq l\leq m,\\
0=&\sum_{l=1}^{m}t_{l,0}e_l(\xi_{j,0})+
\frac{\text{sgn}\left(\sum_{l=1}^{m}t_{l,0}e_{l}(\xi_{j,0})\right)L_0}{\sum_{j=1}^k s_{j,0}^2}s_{j,0}^{\frac12},
\quad 1\leq j\leq k
\endaligned\right.
\end{eqnarray}
for $N=5$, respectively, where
\begin{eqnarray}\label{L_0}
L_0=-\left(\sum_{j=1}^k\left|\sum_{l=1}^{m}t_{l,0}e_{l}(\xi_{j,0})\right|s_{j,0}^{\frac{N-2}2}\right).
\end{eqnarray}
Thus, to solve \eqref{redprob} for the parameters $\pmb{t}$, $\pmb{s}$ and $\pmb{\xi}$ as $\ve\to0^+$, we only need to solve \eqref{redprobA} and \eqref{redprobB} for $N=4$ and $N=5$, respectively, which requires us to construct stable solutions of the limit systems~\eqref{Redp0012} and \eqref{Redp0007}.  We also observe that the system~\eqref{Redp0012} and \eqref{Redp0007} are invariant under the action of $\mathcal{R}\in O(m)$.  Indeed, for every $\mathcal{R}\in O(m)$ and every $(\nu_{1}, \nu_{2}, \cdots, \nu_{m})\in\bbr^m$, let
\begin{eqnarray}\label{Redp1001}
\pmb{\nu}_{*}:=(\nu_{1}^*, \nu_{2}^*, \cdots, \nu_{m}^*)=\mathcal{R}(\nu_{1}, \nu_{2}, \cdots, \nu_{m})
\end{eqnarray}
and
\begin{eqnarray}\label{Redp1002}
(e_1^*, e_2^*, \cdots, e_m^*)=\mathcal{R}(e_1, e_2, \cdots, e_m).
\end{eqnarray}
Then $\{e_i^*\}$ is also an orthogonal system in $\Xi_{\kappa}$ which is the eigenspace according to the eigenvalue $\lambda_{\kappa}$ and
\begin{eqnarray}\label{Redp0039}
\sum_{n=1}^{m}\nu_{n}e_n=\sum_{n=1}^{m}\nu_{n}^*e_n^*,
\end{eqnarray}
which also implies that $\sum_{n=1}^{m}t_{n,0}e_n=\sum_{n=1}^{m}=t_{n,0}^*e_n^*$ and
\begin{eqnarray*}
L_0=-\left(\sum_{j=1}^k\left|\sum_{l=1}^{m}t_{l,0}e_{l}(\xi_{j,0})\right|s_{j,0}^{\frac{N-2}2}\right)=-\left(\sum_{j=1}^k\left|\sum_{l=1}^{m}t_{l,0}^*e_{l}^*(\xi_{j,0})\right|s_{j,0}^{\frac{N-2}2}\right).
\end{eqnarray*}
\subsection{The case $N=4$}
\eqref{Redp0012} is a full coupled system. Our key observation is that the third equation of \eqref{Redp0012} is homogeneous, which implies that it has a solution
\begin{eqnarray}\label{Redp0037}
\pmb{s}_{0}=\left(\left|\sum_{l=1}^{m}t_{l,0}e_{l}(\xi_{1,0})\right|, \left|\sum_{l=1}^{m}t_{l,0}e_{l}(\xi_{2,0})\right|, \cdots, \left|\sum_{l=1}^{m}t_{l,0}e_{l}(\xi_{k,0})\right|\right).
\end{eqnarray}
Under the special choice of \eqref{Redp0037}, the limit system~\eqref{Redp0012} is reduced to
\begin{eqnarray}\label{Redp0038}
\left\{\aligned
0=&\sum_{l=1}^{m}t_{l,0}\frac{\partial e_l(\xi_{j,0})}{\partial x_i},\quad 1\leq j\leq k\text{ and }1\leq i\leq 4,\\
0=&\left\|\sum_{n=1}^{m}t_{n,0}e_n\right\|_2^2\left(\sum_{j=1}^{k}\left(\sum_{n=1}^{m}t_{n,0}e_{n}(\xi_{j,0})\right)e_l(\xi_{j,0})\right)\\
&-\left(\sum_{j=1}^k\left(\sum_{n=1}^{m}t_{n,0}e_{n}(\xi_{j,0})\right)^2\right)\|e_l\|_2^2t_{l,0},\quad 1\leq l\leq m.
\endaligned\right.
\end{eqnarray}
\begin{lemma}\label{LemRedp0005}
Up to a suitable choice of $\mathcal{R}\in O(m)$, \eqref{Redp0038} has a solution $(\pmb{t}_0, \pmb{\xi}_0)\in \left(\bbr\backslash\{0\}\right)^m\times\Omega^k$, which is a solution of the variational problem
\begin{eqnarray*}
\max_{(\pmb{\nu}, \pmb{\eta})\in\mathbb{S}^{m-1}\times\overline{\Omega}}\frac{\sum_{j=1}^k\left(\sum_{l=1}^{m}\nu_{l}e_{l}(\eta_{j})\right)^2}{\left\|\sum_{l=1}^{m}\nu_{l}e_l\right\|_2^2}.
\end{eqnarray*}
In particular, if $m=1$ then $t_{1,0}$ can be taken any number.
\end{lemma}
\begin{proof}
\eqref{Redp0038} is variational since it is the Euler-Lagrange equation of the function
\begin{eqnarray}\label{vbN=4}
\mathcal{H}_*(\pmb{\nu}, \pmb{\eta})=\frac{\sum_{j=1}^k\left(\sum_{l=1}^{m}\nu_{l}e_{l}(\eta_{j})\right)^2}{\left\|\sum_{l=1}^{m}\nu_{l}e_l\right\|_2^2}.
\end{eqnarray}
According to the homogeneity of $\pmb{t}_0$ of \eqref{Redp0038}, we can find critical points of $\mathcal{H}_*(\pmb{\nu}, \pmb{\eta})$ by solving the variational problem
\begin{eqnarray}\label{Redp0040}
c_*=\max_{\mathbb{S}^{m-1}\times\overline{\Omega}^k}\mathcal{H}_*(\pmb{\nu}, \pmb{\eta}).
\end{eqnarray}
Since $\mathbb{S}^{m-1}\times\overline{\Omega}^k$ is compact, $c_*$ is achieved, say $(\pmb{t}_0, \pmb{\xi}_0)$.  Moreover, it is easy to see that $\pmb{\xi}_0\in\Omega^k$ since the eigenfunctions $\{e_l\}\subset C^{\infty}(\overline{\Omega})\cap H^1_0(\Omega)$.  On the other hand, by \eqref{Redp0039},  $\mathcal{H}_*(\pmb{\nu}, \pmb{\eta})$ is invariant under the action of $\mathcal{R}\in O(m)$.  Thus, by acting a suitable $\mathcal{R}\in O(m)$ if necessary, we can always assume that $\pmb{t}_0\in\left(\bbr\backslash\{0\}\right)^m$.  If $m=1$ then we notice that the function \eqref{vbN=4} is independent of $\nu_1$.  Thus, the second equation of \eqref{Redp0038} satisfied by any $t_{1,0}$, provided that $\pmb{\xi}_0$ is the maximum point of \eqref{Redp0040}.
\end{proof}

We are in the position to prove Theorem~\ref{Thm0001}.

\begin{proof}[Proof of Theorems \ref{Thm0001}]
We recall that by Proposition~\ref{fixedpoint} and Lemmas~\ref{cs}, \ref{c1l} and \ref{c3l}, we can complete the proof of Theorem~\ref{Thm0001} by solving \eqref{redprob} for the parameters $\pmb{t}$, $\pmb{s}$ and $\pmb{\xi}$ as $\ve\to0^+$, which, under the assumptions $\beta_j=-\text{sgn}\left(\sum_{l=1}^{m}t_{l,0}e_{l}(\xi_{j,0})\right)$ for all $1\leq j\leq k$ and \eqref{Redp0011}, is equivalent to solving \eqref{redprobA} for the parameters $\pmb{t}$, $\pmb{s}$ and $\pmb{\xi}$ as $\ve\to0^+$.  Moreover, by \eqref{sgnlambda}, we need $\lambda>\lambda_\kappa$ for $N=4$.  Since the arguments will be slightly different for four different situations, we divide the remaining proof into the following four parts.

\vskip0.08in

{\bf Case.~1}\quad $k\geq2$ and $m\geq2$.

\vskip0.04in

We fix
\begin{eqnarray*}
\left\{\aligned
\mu&=\exp\left(-\frac{d_1d_3L_0^2}{\left(d_2\left(\sum_{j=1}^k s_{j,0}^2\right)\left\|\sum_{l=1}^{m}t_{l,0}e_l\right\|_2^2\right)\ve}\right),\\
\tau&=\frac{d_1L_0\exp\left(-\frac{d_1d_3L_0^2}{\left(d_2\left(\sum_{j=1}^k s_{j,0}^2\right)\left\|\sum_{l=1}^{m}t_{l,0}e_l\right\|_2^2\right)\ve}\right)}{\left\|\sum_{l=1}^{m}t_{l,0}e_l\right\|_2^2\ve},
\endaligned\right.
\end{eqnarray*}
where $(\pmb{t}_0, \pmb{\xi}_0)$ is constructed in Lemma~\ref{LemRedp0005} and $\pmb{s}_0$ is given by \eqref{Redp0037}.  We also claim that for any fixed $\pmb{t}_0$ and $\pmb{\xi}_0$,
\begin{eqnarray*}
\left(\text{Ker}\left(\widetilde{\mathbb{M}}_{k\times k}(\pmb{s}_{0})\right)\right)^{\perp}\not=\emptyset,
\end{eqnarray*}
where $\widetilde{\mathbb{M}}_{k\times k}(\pmb{s}_{0})$ is the matrix of the linearization of the third equation of \eqref{Redp0012} at the special solution $\pmb{s}_{0}$ given by \eqref{Redp0037}, $\text{Ker}\left(\widetilde{\mathbb{M}}_{k\times k}(\pmb{s}_{0})\right)$ is its kernel and $\text{Ker}\left(\widetilde{\mathbb{M}}_{k\times k}(\pmb{s}_{0})\right)^{\perp}$ is the orthocomplement of $\text{Ker}\left(\widetilde{\mathbb{M}}_{k\times k}(\pmb{s}_{0})\right)$ in $\mathbb{R}^4$.  Indeed, for any fixed $\pmb{t}_0$ and $\pmb{\xi}_0$, we denote the elements of $\widetilde{\mathbb{M}}_{k\times k}(\pmb{s}_{0})$ by $m_{j,i}$.  By \eqref{Redp0037}, we have
\begin{eqnarray*}
m_{j,j}=\sum_{i=1;l\not=j}^{k}\left|\sum_{l=1}^mt_{l,0}e_{l}(\xi_{i,0})\right|^2
\end{eqnarray*}
for all $1\leq j\leq k$.  Since
\begin{eqnarray*}
0<(k-1)\left(\sum_{l=1}^{k}\left|e_{1}(\xi_{l,0})\right|^2\right)=\sum_{j=1}^{k}m_{j,j}=\text{Trace}\left(\widetilde{\mathbb{M}}_{k\times k}(\pmb{s}_{0})\right),
\end{eqnarray*}
we know that $\left(\text{Ker}\left(\widetilde{\mathbb{M}}_{k\times k}(\pmb{s}_{0})\right)\right)^{\perp}\not=\emptyset$.  It follows from the implicit function theorem that the third equation of \eqref{redprobA} has a solution of the form
\begin{eqnarray}\label{Redp0037}
\pmb{s}_{\ve}=\left(\left|\sum_{l=1}^{m}t_{l,0}e_{l}(\xi_{1,0})\right|+o(1), \left|\sum_{l=1}^{m}t_{l,0}e_{l}(\xi_{2,0})\right|+o(1), \cdots, \left|\sum_{l=1}^{m}t_{l,0}e_{l}(\xi_{k,0})\right|+o(1)\right)
\end{eqnarray}
as $\ve\to0^+$,
where $\pmb{s}_{\ve}-\pmb{s}_{0}\in\left(\text{Ker}\left(\widetilde{\mathbb{M}}_{k\times k}(\pmb{s}_{0})\right)\right)^{\perp}$.  Thus, \eqref{redprobA} is reduced to
\begin{eqnarray}\label{redprobA02}
\left\{\aligned
o(1)=&\sum_{l=1}^{m}t_{l}\frac{\partial e_l(\xi_{j})}{\partial x_i},\quad 1\leq j\leq k\text{ and }1\leq i\leq 4,\\
o(1)=&\left\|\sum_{n=1}^{m}t_{n}e_n\right\|_2^2\left(\sum_{j=1}^{k}\left(\sum_{n=1}^{m}t_{n}e_{n}(\xi_{j})\right)e_l(\xi_{j})\right)\\
&-\left(\sum_{j=1}^k\left(\sum_{n=1}^{m}t_{n}e_{n}(\xi_{j})\right)^2\right)\|e_l\|_2^2t_{l},\quad 1\leq l\leq m
\endaligned\right.
\end{eqnarray}
under the condition~$(\bf{C})$.  Since \eqref{redprobA02} is variational and the limit $(\pmb{t}_0, \pmb{\xi}_0)\in \left(\bbr\backslash\{0\}\right)^m\times\Omega^k$ is a maximum point of the energy function of the limit system by Lemma~\ref{LemRedp0005}, which is stable under small perturbations.  Thus, \eqref{redprobA02} has a solution $(\pmb{t}_\ve, \pmb{\xi}_\ve)\in \left(\bbr\backslash\{0\}\right)^m\times\Omega^k$ as $\ve\to0^+$, which satisfies $(\pmb{t}_\ve, \pmb{\xi}_\ve)\to(\pmb{t}_0, \pmb{\xi}_0)$ as $\ve\to0^+$.

\vskip0.08in

{\bf Case.~2}\quad $k=1$ and $m\geq2$.

\vskip0.04in

In this case, we fix
\begin{eqnarray*}
\mu=\exp\left(-\frac{d_1d_3L_0^2}{\left(d_2s_{1,0}^2\left\|\sum_{l=1}^{m}t_{l,0}e_l\right\|_2^2\right)\ve}\right)
\end{eqnarray*}
and $s_{1}=\left|\sum_{l=1}^{m}t_{l}e_{l}(\xi_{1})\right|$, where $(\pmb{t}_0, \pmb{\xi}_0)$ is constructed in Lemma~\ref{LemRedp0005} and $s_{1,0}$ is given by \eqref{Redp0037}.  Note that the third equation in \eqref{redprobA} is originally reduced from the second equation of \eqref{redprob} under the assumptions $\beta_j=-\text{sgn}\left(\sum_{l=1}^{m}t_{l,0}e_{l}(\xi_{j,0})\right)$ for all $1\leq j\leq k$ and \eqref{Redp0011}, thus, by fixing $\mu$ in \eqref{Redp0011} as above, we can remove the third equation in \eqref{redprobA} by a suitable choice of $\tau$ in \eqref{Redp0011} to directly solve the second equation of \eqref{redprob} as $\ve\to0^+$ for $k=1$.  It follows that we can still safely arrive at the reduced problem~\eqref{redprobA02} for $k=1$ and $m\geq2$, which can be solved in the same way in the case $k\geq2$ and $m\geq2$.

\vskip0.08in

{\bf Case.~3}\quad $k\geq2$ and $m=1$.

\vskip0.04in

In this case, we fix
\begin{eqnarray*}
\tau=\frac{d_1L_0\exp\left(-\frac{d_1d_3L_0^2}{\left(d_2\left(\sum_{j=1}^k s_{j,0}^2\right)\left\|t_{1,0}e_1\right\|_2^2\right)\ve}\right)}{\left\|t_{1,0}e_1\right\|_2^2\ve},
\end{eqnarray*}
where $(t_{1,0}, \pmb{\xi}_0)$ is constructed in Lemma~\ref{LemRedp0005} and $\pmb{s}_0$ is given by \eqref{Redp0037}.
Again, we remark that the second equation in \eqref{redprobA} is originally reduced from the first equation of \eqref{redprob} under the assumptions $\beta_j=-\text{sgn}\left(\sum_{l=1}^{m}t_{l,0}e_{l}(\xi_{j,0})\right)$ for all $1\leq j\leq k$ and \eqref{Redp0011}, thus, by fixing $\tau$ in \eqref{Redp0011} as above, we can remove the second equation in \eqref{redprobA} by a suitable choice of $\mu$ in \eqref{Redp0011} to directly solve the fist equation of \eqref{redprob} as $\ve\to0^+$ for $m=1$.  It follows that the reduced problem is given by
\begin{eqnarray}\label{redprobA01}
\left\{\aligned
o(1)=&\frac{\partial e_l(\xi_{j})}{\partial x_i},\quad 1\leq j\leq k\text{ and }1\leq i\leq 4,\\
o(1)=&\left(\sum_{i=1}^k s_{i}^2\right)t_{1}e_1(\xi_{j})-\left(\sum_{j=1}^k\left|t_{1}e_{1}(\xi_{j})\right|s_{j}\right)\text{sgn}\left(t_{1}e_{1}(\xi_{j})\right)s_{j},
\quad 1\leq j\leq k
\endaligned\right.
\end{eqnarray}
Since $\pmb{\xi}_0$, given by Lemma~\ref{LemRedp0005} for $m=1$, is the maximum point of \eqref{Redp0040}, which is stable under small perturbations, the first equation of \eqref{redprobA01} has a solution $ \pmb{\xi}_\ve\in\Omega^k$ as $\ve\to0^+$, which satisfies $\pmb{\xi}_\ve\to\pmb{\xi}_0$ as $\ve\to0^+$.  Now, the third equation of \eqref{redprobA01} can be solved in the same way in the case $k\geq2$ and $m\geq2$.

\vskip0.08in

{\bf Case.~4}\quad $k=1$ and $m=1$.

\vskip0.04in

In this case, we fix $t_1=1$ and $s_{1}=\left|e_{1}(\xi_{1})\right|$.  Again, we point out that the second and the third equations in \eqref{redprobA} is originally reduced from the first and the second equations of \eqref{redprob} under the assumptions $\beta_j=-\text{sgn}\left(\sum_{l=1}^{m}t_{l,0}e_{l}(\xi_{j,0})\right)$ for all $1\leq j\leq k$ and \eqref{Redp0011}, thus, we can remove the second and the third equations in \eqref{redprobA} by directly solving the first and the second equations of \eqref{redprob} for suitable choices of $\mu$ and $\tau$ in \eqref{Redp0011}, which is equivalent to solving the following linear system:
\begin{eqnarray*}
\left\{\aligned
o(\ve\tau)=&\|e_{1}\|_2^2\ve\tau+d_1\beta_1s_1e_1(\xi_1)\mu,\\
o(\ve\tau)=&d_2\beta_1s_1\mu+d_3t_1e_1(\xi_1)\ve\tau.
\endaligned\right.
\end{eqnarray*}
We point out that the involved matrix can not be totally degenerate since $\|e_1\|_2^2>0$.  Thus, it can be always solved by suitable choices of $\mu$ and $\tau$ in \eqref{Redp0011}.  The finally reduced equation is then given by $o(1)=\frac{\partial e_l(\xi_{1})}{\partial x_i}$, $i=1,2,3,4$, which can be solved in the same way in the case $k\geq2$ and $m=1$.
\end{proof}

\subsection{The case $N=5$}
Since $\tau_{l,0}=t_{l,0}\tau$, by \eqref{Redp0008},  the limits of $\{\tau_{l,0}\}$ and the second equation of \eqref{Redp0007} are all invariant under the scalar multiplication.  Thus, if $\pmb{t}_0=\left(t_{1,0}, t_{2,0},\cdots, t_{m,0}\right)$ solves the second equation of \eqref{Redp0007} then we can find $C>0$ such that $\pmb{t}_*=\left(Ct_{1,0}, Ct_{2,0},\cdots, Ct_{m,0}\right)$ solves the following system:
\begin{eqnarray}\label{Redp0009}
0=t_{l,*}\|e_l\|_2^2-\int_{\Omega}\left|\sum_{n=1}^{m}t_{n,*}e_n\right|^{\frac{4}{3}}\left(\sum_{n=1}^{m}t_{n,*}e_n\right)e_ldx
\end{eqnarray}
for all $1\leq l\leq m$.  On the other hand, if $\pmb{t}_*=\left(t_{1,*}, t_{2,*},\cdots, t_{m,*}\right)$ solves \eqref{Redp0009} then we must have
\begin{eqnarray*}
\left\|\sum_{l=1}^{m}t_{l,*}e_l\right\|_2^2=\left\|\sum_{n=1}^{m}t_{n}e_n\right\|_{2^*}^{2^*}
\end{eqnarray*}
by multiplying \eqref{Redp0009} with $t_{l,*}$ for all $1\leq l\leq m$ and summarizing them from $l=1$ to $l=m$.  It follows that $\pmb{t}_*=\left(t_{1,*}, t_{2,*},\cdots, t_{m,*}\right)$ also solves the second equation of \eqref{Redp0007}.  Based on the above observations, Solving the second equation of \eqref{Redp0007} is equivalent to solving \eqref{Redp0009} up to scalar multiplications.  Thus, we can reduce the solvability of \eqref{Redp0007} to the solvability of the following system:
\begin{eqnarray}\label{Redp1007}
\left\{\aligned
0=&\sum_{l=1}^{m}t_{l,0}\frac{\partial e_l(\xi_{j,0})}{\partial x_i},\quad 1\leq j\leq k\text{ and }1\leq i\leq 5,\\
0=&t_{l,0}\|e_l\|_2^2-\left\langle f\left(\sum_{n=1}^{m}t_{n,0}e_n\right), e_l\right\rangle_{L^2}, 1\leq l\leq m,\\
0=&\sum_{l=1}^{m}t_{l,0}e_l(\xi_{j,0})+
\frac{\text{sgn}\left(\sum_{l=1}^{m}t_{l,0}e_{l}(\xi_{j,0})\right)L_0}{\sum_{j=1}^k s_{j,0}^2}s_{j,0}^{\frac12},
\quad 1\leq j\leq k
\endaligned\right.
\end{eqnarray}
Unlike the case of $N=4$, which we decouple the limit system by constructing a special solution of one of the equations, in the case of $N=5$, the second equation is automatically decoupled in the limit system~\eqref{Redp1007}.  Thus, we can start by constructing stable solutions of the second equation.
\begin{lemma}\label{LemRedp0001}
The second equation of the system~\eqref{Redp1007} has a nondegenerate and nontrivial solution $\pmb{t}_0$ in the sense that $t_{l,0}\not=0$ for all $1\leq l\leq m$, which is also the solution of the variational problem
\begin{eqnarray*}
\max_{\pmb{\nu}\in\bbr^m}\left(\frac{1}{2}\left\|\sum_{n=1}^{m}\nu_{n}e_n\right\|_{2}^{2}-\frac{1}{2^*}\left\|\sum_{n=1}^{m}\nu_{n}e_n\right\|_{2^*}^{2^*}\right).
\end{eqnarray*}
\end{lemma}
\begin{proof}
It is easy to see that the second equation of the system~\eqref{Redp1007} is variational
since it is the Euler-Lagrange equation of the function
\begin{eqnarray}\label{Redp0017}
\mathcal{F}(\pmb{\nu})=\frac{1}{2}\left\|\sum_{n=1}^{m}\nu_{n}e_n\right\|_{2}^{2}-\frac{1}{2^*}\left\|\sum_{n=1}^{m}\nu_{n}e_n\right\|_{2^*}^{2^*}.
\end{eqnarray}
Since $\mathcal{F}(\pmb{\nu}_j)\to-\infty$ as $|\pmb{\nu}_j|\to+\infty$, the variational problem~$\max_{\pmb{\nu}\in\bbr^m}\mathcal{F}(\pmb{\nu})$
is achieved by some $\pmb{t}_0$, where $\{\pmb{\nu}_{1}, \pmb{\nu}_{2}, \cdots, \pmb{\nu}_{m}\}$ is the standard orthogonal system in $\bbr^m$.  Thus,
\begin{eqnarray}\label{Redp0013}
\frac{\partial\mathcal{F}(\pmb{t}_0)}{\partial \nu_{l}}=t_{l,0}\|e_l\|_2^2-\left\langle f\left(\sum_{n=1}^{m}t_{n,0}e_n\right),e_l\right\rangle_{L^2}=0
\end{eqnarray}
for all $1\leq l\leq m$.  Note that $\mathcal{F}(\pmb{\nu})$ is invariant under the action of $\mathcal{R}\in O(m)$ in the sense of \eqref{Redp1001} and \eqref{Redp1002}, thus, we may assume that $\{e_1, e_2, \cdots, e_m\}$ is an orthogonal system in $\Xi_{\kappa}$ such that the Hessian $\mathbb{M}_{m\times m}(\pmb{t}_0)$ is diagonal, where the Hessian of $\mathcal{F}(\pmb{\nu})$ at $\pmb{t}_0$ is given by $\mathbb{M}_{m\times m}(\pmb{t}_0)=\left(\frac{\partial^2\mathcal{F}}{\partial \nu_{n}\partial \nu_{l}}(\pmb{t}_0)\right)_{m\times m}$ with
\begin{eqnarray}\label{Redp0014}
\frac{\partial^2\mathcal{F}(\pmb{t}_0)}{\partial \nu_{n}\partial \nu_{l}}=\delta_{n,l}\|e_l\|_2^2-\frac{7}{3}\left\langle f'\left(\sum_{n=1}^{m}t_{n,0}e_n\right), e_ne_l\right\rangle_{L^2}.
\end{eqnarray}
It remains to prove that $\pmb{t}_0$ is nondegenerate and nontrivial in the sense that $t_{l,0}\not=0$ for all $1\leq l\leq m$.  By \eqref{Redp0013} and \eqref{Redp0014},
\begin{eqnarray*}
\sum_{n=1}^{m}\frac{\partial^2\mathcal{F}(\pmb{t}_{0})}{\partial \nu_{n}\partial \nu_{l}}t_{l,0}&=&t_{l,0}\|e_l\|_2^2-\frac{7}{3}\left\langle f\left(\sum_{n=1}^{m}t_{n,0}e_n\right),e_l\right\rangle_{L^2}\notag\\
&=&-\frac{4}{3}t_{l,0}\|e_l\|_2^2
\end{eqnarray*}
for all $1\leq l\leq m$, which, together with the fact that $\mathbb{M}_{m\times m}(\pmb{t}_0)$ is diagonal, implies that $t_{l,0}\not=0$ for all $1\leq l\leq m$ if and only if $\pmb{t}_0$ is nondegenerate.  Suppose the contrary that $0$ is an eigenvalue of the Hessian $\mathbb{M}_{m\times m}(\pmb{t}_0)$.  Since $1$ is a strictly global maximum point of the function $F(\varrho)=\mathcal{F}(\varrho\pmb{t}_0)$,
without loss of generality, we may assume that $\pmb{\nu}_{1}, \pmb{\nu}_{2}, \cdots, \pmb{\nu}_{m'}$ are the kernel of the Hessian $\mathbb{M}_{m\times m}(\pmb{t}_0)$ with $m'\leq m-1$.  By the Taylor expansion and the fact that $\mathbb{M}_{m\times m}(\pmb{t}_0)$ is diagonal,
\begin{eqnarray*}
\mathcal{F}(\pmb{\nu})=\mathcal{F}(\pmb{t}_0)+\sum_{l=m'+1}^{m}\frac{\partial^2\mathcal{F}(\pmb{t}_{0})}{\partial \nu_{l}^2}\frac{(\nu_l-t_{l,0})^2}{2}+\mathcal{T}+o\left(\left|\overrightarrow{\nu}-\overrightarrow{t}_0\right|^3\right)
\end{eqnarray*}
where
\begin{eqnarray*}
\mathcal{T}=\sum_{n,l,j=1}^{m}\frac{\partial^3\mathcal{F}(\pmb{t}_{0})}{\partial \nu_{n}\partial \nu_{l}\partial \nu_{j}}\frac{(\nu_n-t_{n,0})(\nu_l-t_{l,0})(\nu_j-t_{j,0})}{6}.
\end{eqnarray*}
Since $\pmb{t}_{0}$ is a global maximum point and $\pmb{\nu}_{1}, \pmb{\nu}_{2}, \cdots, \pmb{\nu}_{m'}$ are the eigenfunctions of $0$ with $m'\leq m-1$, we must have $\frac{\partial^3\mathcal{F}(\pmb{t}_{0})}{\partial \nu_{n}\partial \nu_{l}^2}=0$
for all $1\leq l\leq m'$ and all $1\leq n\leq m$, which implies that
\begin{eqnarray*}
0=-\frac{28}{9}\left\langle f''\left(\sum_{j=1}^{m}t_{j,0}e_j\right), e_ne_l^2\right\rangle_{L^2}
\end{eqnarray*}
for all $1\leq l\leq m'$ and all $1\leq n\leq m$.  It follows that
\begin{eqnarray}\label{Redp0016}
0=\left\langle f'\left(\sum_{j=1}^{m}t_{j,0}e_j\right), e_l^2\right\rangle_{L^2}
\end{eqnarray}
for all $1\leq l\leq m'$.
On the other hand, since $\pmb{\nu}_{1}, \pmb{\nu}_{2}, \cdots, \pmb{\nu}_{m'}$ are the kernel of the Hessian $\mathbb{M}_{m\times m}(\pmb{t}_0)$, by \eqref{Redp0014},
\begin{eqnarray*}
0=\left\langle\left(1-\frac{7}{3}f\left(\sum_{j=1}^{m}t_{j,0}e_j\right)\right), e_l^2\right\rangle_{L^2}
\end{eqnarray*}
for all $1\leq l\leq m'$, which contradicts \eqref{Redp0016}.  Thus, $0$ can not be an eigenvalue of the Hessian $\mathbb{M}_{m\times m}(\pmb{t}_0)$ and $\pmb{t}_0$ is nondegenerate, which also implies that $\pmb{t}_0$ is nontrivial, that is, $t_{l,0}\not=0$ for all $1\leq l\leq m$.
\end{proof}

We next construct stable solutions of the first equation of the limit system~\eqref{Redp1007}.
\begin{lemma}\label{LemRedp0002}
Let $\pmb{t}_0$ be a nondegenerate and nontrivial solution of the second equation of the system~\eqref{Redp1007} constructed by Lemma~\ref{LemRedp0001}.  Then
the first equation of \eqref{Redp1007} has at least $n_\kappa$ solutions $\xi_{1,0}, \xi_{2,0}, \cdots, \xi_{n_0,0}$ where $\xi_{j,0}$ is either a local maximum point or a local minimum point of the function $\mathcal{G}(x)=\sum_{l=1}^{m}t_{l,0}e_l(x)$
and $n_\kappa$ is the number of the nodal domains of $\mathcal{G}(x)$.  Moreover, $\left(\text{Ker}\left(\mathbb{G}(\xi_{j,0})\right)\right)^{\perp}\not=\emptyset$
for all $1\leq j\leq n_\kappa$, where $\text{Ker}\left(\mathbb{G}(\xi)\right)$ is the kernel of the Hessian $\mathbb{G}(\xi_{j,0})=\left(\frac{\partial^2\mathcal{G}(\xi_{j,0})}{\partial x_i\partial x_j}\right)_{5\times5}$
and $\left(\text{Ker}\left(\mathbb{G}(\xi_{j,0})\right)\right)^{\perp}$ is its orthogonal complement in $\bbr^5$.
\end{lemma}
\begin{proof}
The first equation of \eqref{Redp1007} is also variational since it is the Euler-Lagrange equation of the function
\begin{eqnarray*}
\mathcal{G}(x)=\sum_{l=1}^{m}t_{l,0}e_l(x).
\end{eqnarray*}
Note that $\mathcal{G}(x)$ is also an eigenfunction of the Laplacian operator $-\Delta$ in $L^2(\Omega)$ with the Dirichlet boundary according to the eigenvalue $\lambda_{\kappa}$, thus,
$\mathcal{G}(x)$ has at least $n_\kappa$ stable critical points $\xi_{1,0}, \xi_{2,0}, \cdots, \xi_{n_\kappa,0}$ where $\xi_{j,0}$ is either a local maximum point or a local minimum point of $\mathcal{G}(x)$ with $\mathcal{G}(\xi_{j,0})\not=0$ for all $1\leq j\leq n_\kappa$.  It follows that
\begin{eqnarray*}
-\Delta\mathcal{G}(\xi_{j,0})=\lambda_{\kappa}\mathcal{G}(\xi_{j,0})\not=0
\end{eqnarray*}
for all $1\leq j\leq n_\kappa$.
Since $\Delta\mathcal{G}(\xi_{j,0})=\text{Trace}\left(\mathbb{G}(\xi_{j,0})\right)$ where $\mathbb{G}(\xi_{j,0})$ is the Hessian of $\mathcal{G}(x)$ at $\xi_{j,0}$,
we know that $\left(\text{Ker}\left(\mathbb{G}(\xi_{j,0})\right)\right)^{\perp}\not=\emptyset$ for all $1\leq j\leq n_\kappa$.
\end{proof}

We finally construct stable solutions of the third equation of the system~\eqref{Redp1007}.

\begin{lemma}\label{LemRedp0003}
Let $1\leq k\leq n_\kappa$, $\pmb{t}_0$ be a nondegenerate and nontrivial solution of the second equation of the system~\eqref{Redp1007} constructed by Lemma~\ref{LemRedp0001} and $\xi_{1,0}, \cdots, \xi_{k,0}$ is the solutions of the first equation of \eqref{Redp1007} constructed in Lemma~\ref{LemRedp0002}, where $n_\kappa$ is the number of the nodal domains of the function $\sum_{l=1}^{m}t_{l,0}e_l(x)$.  Then
the third equation of the system~\eqref{Redp1007} has a solution $\pmb{s}_0$, which is the solution of the variational problem
\begin{eqnarray*}
\max_{\pmb{\nu}\in(\bbr_+)^k}\left(\sum_{j=1}^{k}\left|\sum_{l=1}^{m}t_{l,0}e_l(\xi_{j,0})\right|^2\frac{\nu_{j}^2}{2}-
\frac{1}{6}\left(\sum_{j=1}^k\left|\sum_{l=1}^{m}t_{l,0}e_{l}(\xi_{j,0})\right|\nu_{j}^{3}\right)^2\right),
\end{eqnarray*}
where $(\bbr_+)^k=\{\pmb{\nu}\in\bbr^k\mid \nu_j>0\text{ for all }1\leq j\leq k\}$.
\end{lemma}
\begin{proof}
Let $s_{j,*}=s_{j,0}^{\frac12}$ for all $1\leq j\leq k$.  Then by \eqref{L_0}, the solvability of the third equation of the system~\eqref{Redp1007} is equivalent to the solvability of the following equation:
\begin{eqnarray}\label{Redp0017}
\left\{\aligned&0=\left|\sum_{l=1}^{m}t_{l,0}e_l(\xi_{j,0})\right|s_{j,*}-\left(\sum_{n=1}^k\left|\sum_{l=1}^{m}t_{l,0}e_{l}(\xi_{n,0})\right|s_{n,*}^3\right)s_{j,*}^2,\\
&s_{j,*}>0,\quad 1\leq j\leq k.
\endaligned\right.
\end{eqnarray}
We claim that \eqref{Redp0017} is variational by proving that it is the Euler-Lagrange equation of the function
\begin{eqnarray*}\label{Redp0025}
\mathcal{H}(\pmb{\nu})=\sum_{j=1}^{k}\left|\sum_{l=1}^{m}t_{l,0}e_l(\xi_{j,0})\right|^2\frac{\nu_{j}^2}{2}-
\frac{1}{6}\left(\sum_{j=1}^k\left|\sum_{l=1}^{m}t_{l,0}e_{l}(\xi_{j,0})\right|\nu_{j}^{3}\right)^2.
\end{eqnarray*}
Indeed, it is easy to see that if $\pmb{s}_0=(s_{1,0},s_{2,0},\cdots,s_{k,0})$ solves \eqref{Redp0017} then $\pmb{s}_0$ is a critical point of $\mathcal{H}(\pmb{\nu})$.  On the other hand, if $\pmb{s}_0=(s_{1,0},s_{2,0},\cdots,s_{k,0})$ is a critical point of $\mathcal{H}(\pmb{\nu})$ then $\pmb{s}_0$ satisfies
\begin{eqnarray}\label{Redp0019}
0=\left|\sum_{l=1}^{m}t_{l,0}e_l(\xi_{j,0})\right|s_{j,0}-\left(\sum_{n=1}^k\left|\sum_{l=1}^{m}t_{l,0}e_{l}(\xi_{n,0})\right|s_{n,0}^3\right)s_{j,0}^2
\end{eqnarray}
for all $1\leq j\leq k$.
If there exists $1\leq j\leq k$ such that $s_{j,0}<0$, then by \eqref{Redp0019}, we have $\sum_{n=1}^k\left|\sum_{l=1}^{m}t_{l,0}e_{l}(\xi_{n,0})\right|s_{n,0}^3<0$ which, together with \eqref{Redp0019} once more, implies that $s_{i,0}<0$ for all $1\leq i\leq k$.  It follows that $-\pmb{s}_0=(-s_{1,0},-s_{2,0},\cdots,-s_{k,0})$ is a solution of \eqref{Redp0017} by its oddness.  We can find nontrivial critical points of $\mathcal{H}(\pmb{\nu})$ by solving the variational problem
\begin{eqnarray}\label{Redp0020}
c=\max_{\pmb{\nu}\in(\bbr_+)^k}\mathcal{H}(\pmb{\nu}).
\end{eqnarray}
For the sake of simplicity, we denote $\left|\sum_{l=1}^{m}t_{l,0}e_l(\xi_{j,0})\right|$ by $a_j$ for all $1\leq j\leq k$.  It is easy to see that
\begin{eqnarray}\label{Redp0021}
\frac{\partial \mathcal{H}(\pmb{\nu})}{\partial \nu_j}=\left|\sum_{i=1}^{m}t_{i,0}e_i(\xi_{j,0})\right|^2\nu_j-\left(\sum_{l=1}^{k}\left|\sum_{i=1}^{m}t_{i,0}e_i(\xi_{l,0})\right|v_l^3\right)\left|\sum_{i=1}^{m}t_{i,0}e_i(\xi_{j,0})\right|v_j^2.
\end{eqnarray}
Since $\mathcal{H}(\pmb{\nu}_j)\to-\infty$ as $|\pmb{\nu}_j|\to+\infty$ for all $1\leq j\leq k$, by \eqref{Redp0020} and \eqref{Redp0021},
\begin{eqnarray*}
c=\max_{\pmb{\nu}\in(\bbr_+)^k}\mathcal{H}(\pmb{\nu})=\max_{\pmb{\nu}\in(\bbr_+)^k\backslash\partial\left((\bbr_+)^k\right)}\mathcal{H}(\pmb{\nu})
\end{eqnarray*}
is attained by some $\pmb{s}_{*,0}\in(\bbr_+)^k$, where $\{\pmb{\nu}_{1}, \pmb{\nu}_{2}, \cdots, \pmb{\nu}_{k}\}$ is the standard orthogonal system in $\bbr^k$ and $\partial\left((\bbr_+)^k\right)$ is the boundary of $(\bbr_+)^k$.
\end{proof}

We are in the position to prove Theorem~\ref{Thm0002}.

\begin{proof}[Proof of Theorems \ref{Thm0002}]
As that in the case of $N=4$, we recall that by Proposition~\ref{fixedpoint} and Lemmas~\ref{cs}, \ref{c1l} and \ref{c3l}, we can complete the proof of Theorem~\ref{Thm0002} by solving \eqref{redprob} for the parameters $\pmb{t}$, $\pmb{s}$ and $\pmb{\xi}$ as $\ve\to0^+$, which, under the assumptions $\beta_j=-\text{sgn}\left(\sum_{l=1}^{m}t_{l,0}e_{l}(\xi_{j,0})\right)$ for all $1\leq j\leq k$ and \eqref{Redp0008}, is equivalent to solving \eqref{redprobB} for the parameters $\pmb{t}$, $\pmb{s}$ and $\pmb{\xi}$ as $\ve\to0^+$.  Moreover, by \eqref{sgnlambda}, we need $\lambda<\lambda_\kappa$ for $N=5$.  We fix
\begin{eqnarray*}
\left\{\aligned
\tau&=\left(\frac{\left\|\sum_{l=1}^{m}t_{l,0}e_l\right\|_2^2}{\left\|\sum_{n=1}^{m}t_{n,0}e_n\right\|_{2^*}^{2^*}}\right)^{\frac{3}{4}}\ve^{\frac{3}{4}},\\
\mu&=\left(\frac{d_3L_0}{ d_2\left(\sum_{j=1}^k s_{j,0}^2\right)}\right)^2\left(\frac{\left\|\sum_{l=1}^{m}t_{l,0}e_l\right\|_2^2}{\left\|\sum_{n=1}^{m}t_{n,0}e_n\right\|_{2^*}^{2^*}}\right)^{\frac{3}{2}}\ve^{\frac{3}{2}},
\endaligned\right.
\end{eqnarray*}
where $\pmb{t}_0$, $\pmb{\xi}_0$ and $\pmb{s}_0$ are constructed in Lemmas~\ref{LemRedp0001}, \ref{LemRedp0002} and \ref{LemRedp0003}, respectively.  Since $\pmb{t}_0$, $\pmb{\xi}_0$ and $\pmb{s}_0$ are either a local maximum point or a local minimum point, which are stable under small perturbations, we can use the variational arguments to solve \eqref{redprobB} by firstly solving the second equation, next solving the first equation and finally solving the third equation in the same way in the proofs of Lemmas~\ref{LemRedp0001}, \ref{LemRedp0002} and \ref{LemRedp0003}.
\end{proof}

\end{document}